\documentclass[11pt]{article}
\usepackage{geometry}
\usepackage{amsmath,amssymb,amsthm,makeidx,dsfont}
\usepackage{tikz}
\usepackage{float}
\usepackage{subcaption}
\usetikzlibrary{arrows.meta,calc}
\usepackage{charter}
\usepackage{xcolor}
\usepackage[nottoc]{tocbibind} %adds references to TOC
\usepackage[pagebackref=true]{hyperref}
\renewcommand*\backref[1]{\ifx#1\relax \else \mbox{\textcolor{gray}{Page #1}} \fi}
\usepackage{cite}
\hypersetup{
	colorlinks,%
	citecolor=blue,%
	filecolor=black,%
	urlcolor=black,%
	linkcolor=blue
}

\newcommand\blfootnote[1]{%
	\begingroup
	\renewcommand\thefootnote{}\footnote{#1}%
	\addtocounter{footnote}{-1}%
	\endgroup
}

\newtheorem{theo}{Theorem}[section]
\newtheorem{prop}[theo]{Proposition}
\newtheorem{lem}[theo]{Lemma}
\newtheorem{cor}[theo]{Corollary}

\theoremstyle{definition} %coloca o título em negrito e o corpo normal

\theoremstyle{definition}

\newcommand{\N}{\mathbb{N}}

\newcommand{\R}{\mathbb{R}}

\usepackage{mathtools}

\newcommand{\fconvfk}{{\mathrm{Conv}(\R^k; \R)}} %finite convex functions on R^k
\newcommand{\fconvf}{{\mathrm{Conv}(\R^n; \R)}} %finite convex functions on R^n
\newcommand{\fconvfz}{{\mathrm{Conv}_{(o)}(\R^n; \R)}} %finite convex functions on R^n, C^2 in a neighborhood of the origin

\newcommand{\s}{\mathbb{S}}
\newcommand{\sn}{\mathbb{S}^{n-1}}

\newcommand{\K}{\mathcal{K}}
\newcommand{\Kn}{\K^n}

\newcommand{\cR}{\operatorname{\mathcal R}}

\newcommand{\oZ}{\operatorname{Z}}
\newcommand{\oz}{\operatorname{z}}
\newcommand{\oV}{\operatorname{V}}

\newcommand{\ot}{\operatorname{t}}

\newcommand{\MA}{\mathrm{MA}} %Monge-Ampere measure

\DeclareMathOperator{\SO}{\operatorname{SO}}
\newcommand{\SOn}{\SO(n)}

\newcommand{\hm}{\mathcal H} %Hausdorff measure
\renewcommand{\d}{\,\mathrm{d}} %Integration
\newcommand{\Hess}{{\operatorname{D}}^2} %Hessian

\title{Additive Kinematic Formulas for Functional Minkowski Vectors}
\author{Mohamed A. Mouamine}
\date{}

\begin{document}

\maketitle
\begin{abstract}
		We establish an additive kinematic formula for the functional Minkowski vectors using mixed Monge–Amp\`ere measures. These vectors, recently introduced and characterized by the author and F. Mussnig, form a natural family of vector-valued valuations on the space of convex functions. This result represents the first integral geometric application of this characterization.

	\end{abstract}
	
\section{Introduction}
    Let $\Kn$ denote the set of \emph{convex bodies} in $\R^n$, that is, the set of non-empty, compact, and convex subsets of $\R^n$ where $n \in \N$. One of the most relevant functionals defined on $\Kn$ is the volume of a convex body $K$, meaning its $n$-dimensional Lebesgue measure that we denote by $V_n(K)$.
    
    Let $B^n$ denote the $n$-dimensional Euclidean unit ball in $\R^n$. For $\lambda > 0$, we have the following Steiner formula
        \begin{align*}
            V_n(K + \lambda B^n) = \sum_{j = 0}^n \lambda^{n-j}\kappa_{n-j} V_j(K)
        \end{align*}
    where $\kappa_k$ is the $k$-dimensional volume of the unit ball in $\R^k$ and 
        \begin{align*}
            K + L = \{ x + y ; \ x \in K \ \text{and} \ y \in L\}
        \end{align*}
    denotes the Minkowski sum of $K$ and $L$. The functionals $V_j$ are called the intrinsic volumes. An important formula involving the intrinsic volumes in integral geometry is given in the following theorem, which is known as the additive kinematic formula.
    \begin{theo}[\cite{hug_weil_2020}, Theorem 5.13]\label{additive_kinematic_formula_bodies}
        Let $0 \leq j \leq n$ and $K, L \in \Kn$. Then 
        \begin{align*}
            \int_{\SOn} V_j(K + \vartheta L ) \d\vartheta = \sum_{k = 0}^j \frac{\binom{2n-j}{n-j}\kappa_{n-k} \kappa_{n+k-j}}{\binom{2n-j}{n-k}\kappa_n \kappa_{n-j}} V_k(K)V_{j-k}(L),
        \end{align*}
         where we denote by $\SOn$ the group of special orthogonal transformations on $\R^n$, and $\d\vartheta$ denotes integration with respect to the Haar probability measure on $\SO(n)$.
    \end{theo}
     An elegant proof of this result relies on Hadwiger's characterization theorem, which is given below, through the notion of a \emph{valuation}. We say that a map $\oZ : \Kn \rightarrow \R$ is a valuation if the following holds for every $K, L \in \Kn$ such that $ K \cup L \in \Kn$,
    \begin{align*}
        \oZ(K \cup L) + \oZ(K \cap L) = \oZ(K) + \oZ(L).
    \end{align*} 

    \begin{theo}[\cite{hadwiger}, Satz IV, \cite{schneider_cb}, Theorem 6.4.14]\label{hadwiger_theorem}
       A map $\oZ: \Kn \rightarrow \R$ is a continuous, translation and rotation invariant valuation if and only if there exists $c_0,\dots,c_n \in \R$ such that 
       \begin{align*}
           \oZ(K) = \sum_{j=0}^n c_j V_j(K)
       \end{align*}
  for every $K \in \Kn$.
    \end{theo}
    Here, the continuity of operators on $\Kn$ is understood to be with respect to the Hausdorff metric; see \cite[Section 1.8]{schneider_cb}. 
    The map $\oZ$ is said to be translation invariant if $\oZ(K + x) = \oZ(K)$ for all $K \in \Kn$ and for $x \in \R^n$ and is said to be rotation invariant if $\oZ(\vartheta K) = \oZ(K)$ for all $K \in \Kn$ and all $\vartheta \in \SOn$. 

    We point out that other additive kinematic formulas were derived for more general valuations in \cite{bernig_hug_2018, hug_schneider, hug_mussnig_ulivelli_cjm} and for different groups that act on the space of convex bodies in \cite{bernig_fu_2011, brauner_hofstaetter_ortega-moreno_zonal}. For general references on integral geometry, see \cite{schneider_weil,klain_rota}.
    
    This article aims to establish an additive kinematic formula for the functional Minkowski vectors introduced in \cite{mouamine_mussnig_1} and characterized in \cite{mouamine_mussnig_2}. Let 
    \begin{align*}
        \fconvf\coloneqq \{v : \R^n \rightarrow \R : v\ \text{is convex}\}
    \end{align*}
    denote the space of (finite-valued) convex functions on $\R^n$. We equip this space and its subspaces with the topology of \emph{epi-convergence} and note that in this setting, this topology coincides with the topology of pointwise convergence. 
    We observe that the support function of a convex body $K \in \Kn$, denoted by $h_K$, is an element of $\fconvf$, and it is given by 
    \begin{align*}
     h_K(x) = \max_{y \in K} \langle x,y\rangle
    \end{align*}
    where $\langle \cdot,\cdot \rangle $ denotes the standard inner
    product on $\R^n$. We remark that $h_{B^n}(x) = |x|$ where $|x|$ denotes the Euclidean norm of $x \in \R^n$.

    Let $C_b((0,\infty))$ denote the space of continuous functions with bounded support on $(0,\infty)$.
    If $\alpha \in C_b((0,\infty))$ satisfies
    \[\lim_{s \rightarrow 0^+}\alpha(s)s = 0,\]
    then the functional Minkowski vectors are operators of the form
    \begin{align*}
        v \rightarrow \int_{\R^n}\alpha(|x|)x \d\MA(v[j],h_{B^n}[n-j];x)
    \end{align*}
    where $v \in \fconvf$, and $\alpha \in C_b((0,\infty)) $ is such that $\lim_{s \rightarrow 0^+}\alpha(s)s = 0$. Here, $\MA(v_1,\dots,v_n;\cdot)$ denotes the \emph{mixed Monge--Amp\`ere measure} of the functions $v_1,\ldots,v_n \in \fconvf$, and the notation $v[j]$ means that the function $v$ is repeated $j$ times, while $h_{B^n}[n-j]$ means that $h_{B^n}$ is repeated $n-j$ times. Under $C^2$ assumptions on the entries, the mixed Monge--Amp\`ere  measure is absolutely continuous with respect to the Lebesgue measure and
    \begin{align*}
        \d\MA(v_{1},\dots,v_{n}; x) = \det(\Hess  v_1(x),\dots,\Hess v_n(x))\d x
    \end{align*}
    where $\Hess v_i(x)$ denotes the Hessian matrix of $v_i$ at $x$, for $i \in \{1,\dots,n\}$ and where 
    \[\det: (\R^{n\times n})^n \rightarrow \R\]
    denotes the mixed discriminant. We use the abbreviated notation
    
    \begin{align}\label{monge_ampere_j}
        \MA_j(v;\cdot) \coloneqq\MA(v[j],h_{B^n}[n-j];\cdot).  
    \end{align}
    
    The functional Minkowski vectors vanish identically on support functions of convex bodies, that is, if $K \in \Kn$, then for $j \in \{0,\dots,n-1\}$,
    \begin{align}\label{minkowski_relations}
        \int_{\R^n}\alpha(|x|)x \d\MA(h_K[j],h_{B^n}[n-j];x) \simeq \int_{\sn} z \d S_j(K,z)= 0
    \end{align}
    where the last equality is the Minkowski relation and where $\simeq$ denotes equality up to constant factors, depending on $n$, $j$, and $\alpha$, and $S_j(K,\cdot)$ denotes the $j$-th area measure of $K$; see \cite[(5.30) and (5.52)]{schneider_cb}, and \cite[(4.3)]{mouamine_mussnig_2}.

    We are now ready to state the main theorem of this article. 
\begin{theo}\label{main_theorem_monge_ampere}
    If $0 \leq j \leq n$ and $\alpha \in C_b((0,\infty))$ is such that $\underset{s \rightarrow 0^+}{\lim} \alpha(s)s = 0$, then 
    \begin{equation}\label{additive_kinematic_formula_vector}
    \begin{aligned}
         &\kappa_n \int_{\SOn} \int_{\R^n}\alpha(|y|) y \d\MA_j(u+(v\circ \vartheta^{-1});y ) \d\vartheta   \\
        &=\sum_{k=1}^{j}\binom{j}{k}\int_{\R^n}\int_{\R^n}\alpha\left(\max\{|x|,|y|\}\right)x\d\MA_{j-k}(v;y)\d\MA_k(u;x)
    \end{aligned}
    \end{equation}
    for every $u,v \in \fconvf$.
\end{theo}
    Let us point out that the right-hand side of (\ref{additive_kinematic_formula_vector}) is well defined. Indeed, for fixed $y \in \R^n$, set
    \[\alpha_y(s):=\alpha(\max\{s,|y|\}),\qquad s>0.\]
    Then $\alpha_y\in C_b((0,\infty))$ and $\lim_{s\to 0+}s\alpha_y(s)=0$. Thus, by \cite[Lemma 3.4]{mouamine_mussnig_1}, the map
   
    \[x\mapsto \alpha_y(|x|)x=\alpha(\max\{|x|,|y|\})x,\qquad x\neq0,\]
    has a unique continuous extension to $\R^n$, and this extension
    takes the value $0$ at the origin. 
    
    We bring the reader's attention to the fact that D. Hug, F. Mussnig, and J. Ulivelli \cite{hug_mussnig_ulivelli_cjm} established an additive kinematic formula for functional intrinsic volumes on $\fconvf$ which serves as a functional analog to Theorem \ref{additive_kinematic_formula_bodies}. It is given in the following theorem.
\begin{theo}[\cite{hug_mussnig_ulivelli_cjm}, Theorem 1.4]
    If $0 \leq j \leq n$ and $\alpha: [0,\infty) \rightarrow [0,\infty)$ is a measurable function, then 
    \begin{equation}\label{additive_kinematic_formula_scalar}
        \begin{aligned}
             \kappa_n &\int_{\SOn} \int_{\R^n} \alpha(|y|) \d\MA_j(u + (v\circ\vartheta^{-1}); y) \d\vartheta \\
        &= \sum_{i=0}^{j}\binom{j}{i}\int_{\R^n}\int_{\R^n}\alpha(\max\{|x|,|y|\})\d\MA_{j-i}(v; y)\d\MA_i(u; x)
        \end{aligned}
    \end{equation}
   for every $u,v \in \fconvf$.
\end{theo}
Although both formulas (\ref{additive_kinematic_formula_vector}) and (\ref{additive_kinematic_formula_scalar}) establish additive kinematic formulas on $\fconvf$, they differ in two important aspects. First, formula (\ref{additive_kinematic_formula_vector}) deals with vector-valued functionals, while formula (\ref{additive_kinematic_formula_scalar}) deals with scalar-valued functionals. Secondly, one can observe that the expression in (\ref{additive_kinematic_formula_scalar}) is symmetric in the functions $u$ and $v$, while formula (\ref{additive_kinematic_formula_vector}) is rotation invariant in $v$ but rotation equivariant in $u$; see Section \ref{proof_main_theorem}.

If we choose $u=h_K$ for $K\in\Kn$ in Theorem \ref{main_theorem_monge_ampere}, then (\ref{minkowski_relations}) shows that (\ref{additive_kinematic_formula_vector}) vanishes. By contrast, substituting the support function of $K$ for $v$ yields the following result.
\begin{cor}\label{substituting_expression}
        If $K \in \Kn$, and $\alpha \in C_b((0,\infty))$ is such that $\underset{s \rightarrow 0^+}{\lim} \alpha(s)s = 0$, then 
        \begin{align*}
            \kappa_n \int_{\SOn} \int_{\R^n}\alpha(|y|) &y \d\MA_j(u+(h_K\circ \vartheta^{-1});y ) \d\vartheta \\&=  \sum_{k=1}^{j}\frac{\binom{j}{k}}{\binom{n}{j-k}}\kappa_{n-j+k}V_{j-k}(K)\int_{\R^n}\alpha(|x|)x\d\MA_k(u;x)
        \end{align*}
    for every $u \in \fconvf$.
    \end{cor}
\section{Preliminaries}
    We work in $n$-dimensional Euclidean space $\R^n$, with $n \geq 1$, endowed with the Euclidean norm $|\cdot|$ and the usual scalar product $\langle \cdot, \cdot \rangle$. We write $e_1,\ldots, e_n$ for the canonical basis vectors of $\R^n$ and use the coordinates, $x = (x_1,\ldots, x_n)$, for $x  \in \R^n$ with respect to this basis.
    
    Let $B^n := \{x \in \R^n : |x| \leq 1\}$ be the Euclidean unit ball and $\sn$ the unit sphere in $\R^n$. We denote by $\kappa_n$ the $n$-dimensional volume of $B^n$ and by $\omega_n$ the $(n-1)$-dimensional Hausdorff measure, that we denote by $\hm ^{n-1}$, of $\sn$. In this section, we will collect some results on (finite-valued) convex functions. We refer to \cite{rockafellar,rockafellar_wets} for general references.
    
\subsection*{Convex functions} Let $ \partial v(x)$ denote the subdifferential of the function $ v\in \fconvf$ at $x$ given by 
    \begin{align*}
        \partial v(x) = \{ y \in \R^n : v(z) \geq v(x) + \langle y, z-x \rangle  \  \forall z \in \R^n\}.
    \end{align*}
    This set is closed and convex for any $x \in \R^n$, and the function $v$ is differentiable if and only if the set $\partial v(x)$ contains exactly one element, namely the gradient of $v$ at $x$, that is $\nabla v(x)$. We recall that, if $u$ and $v$ are elements of $\fconvf$, then, for every $x \in \R^n$,
     
     \begin{equation}\label{singular_point_addition}
          \partial(u+ v)(x) = \partial u(x) + \partial v(x).
     \end{equation}
    
     See, for instance, \cite[Theorem 23.8]{rockafellar}.
     
     For finite-valued convex functions, epi-convergence agrees with pointwise convergence and addition is continuous with respect to this topology. In particular, if \(u_i\to u\) and \(v_i\to v\) epi-converge in $\fconvf$ then \(u_i+v_i\to u+v\) epi-converges; see \cite[Theorem 7.46]{rockafellar_wets}. 
     
    Another result that will be needed is the following.
     \begin{lem}[\cite{colesanti_ludwig_mussnig_6}, Lemma 3.3]\label{JCM}
        The map 
        \begin{align*}
            (\vartheta,v) \rightarrow v \circ \vartheta^{-1}
        \end{align*}
        is jointly continuous on $\SOn \times \fconvf$.
    \end{lem}
\subsection*{Hessian measures} To each $ v \in \fconvf$, we assign a family of nonnegative, locally finite Borel measures that we denote by $\Phi^n_j(v;\cdot)$, where $j \in \{0, \ldots, n\}$, through the relation 
    \begin{align*}
        \hm ^{n}( \{x + sy : x \in B, y \in \partial v(x)\}) = \sum_{j=0}^{n} s^j \Phi^n_j(v;B)
    \end{align*}
    for every Borel set $ B \subset \R^n$ and every $s\geq 0 $. They are called the \emph{Hessian measures} associated with $v$. Furthermore, if $v$ is $C^2$ on some open set $A \subset \R^n$, then $\Phi^n_j(v;\cdot)$ is absolutely continuous with respect to the Lebesgue measure and
    \begin{align}\label{hessian_abs_cont_lebesgue}
        \d\Phi^n_j(v;x) = [\Hess  v(x)]_j \d x
    \end{align}
    on the open set $A$ for all $0 \leq j \leq n$. Here, $\Hess v$ denotes the Hessian matrix of $v$ at $x \in \R^n$.If \(A\in\mathbb R^{n\times n}\) is symmetric and has eigenvalues \(\lambda_1,\ldots,\lambda_n\), we set
    \[[A]_0:=1, \qquad [A]_j:=\sum_{1\leq i_1<\cdots<i_j\leq n}\lambda_{i_1}\cdots\lambda_{i_j},\qquad 1\leq j\leq n.
    \]
   
    We refer the reader to \cite{colesanti_ludwig_mussnig_3, colesanti_1997, trudinger_wang_hessian_ii} for more results and details on Hessian measures. In \cite{knoerr_monge_ampere}, it was established that the Hessian measures are locally determined, that is, if $u, v \in \fconvf$ are such that $u \equiv v$ on some open set $A \subset \R^n$, then 
        \begin{align}\label{locally_determined}
            \Phi^n_j(u; A \cap B) = \Phi^n_j(v; A \cap B)
        \end{align}
    for every relatively compact Borel set $B \subset \R^n$.
    
 We say that a map $\oZ:\fconvf \rightarrow  \R $ is a valuation if 
    \begin{align*}
    \oZ( v \vee w ) + \oZ(v \wedge w ) = \oZ(v) + \oZ(w) 
    \end{align*}
    for all $v, w \in \fconvf$ such that their pointwise maximum $ v \vee w$ and their pointwise minimum $ v \wedge w$ are elements of $\fconvf$.
    
     A valuation $\oZ$ is called \emph{dually epi-translation invariant} if
     \[\oZ(v+ f )= \oZ(v) \quad \text{for every }\  v \in \fconvf,\ \text{for every affine map}\ f: \R^n \rightarrow \R.\] 
     It is said to be \emph{rotation invariant} if 
     
     \[\oZ(v\circ \vartheta^{-1}) = \oZ(v) \quad \text{for every}\ v \in \fconvf,\ \text{for every}\  \vartheta \in \SOn,\]
     and a valuation $\ot$ is said to be \emph{rotation equivariant} if \[\ot (v \circ \vartheta^{-1}) = \vartheta \ot(v)\quad \text{for every}\ v \in \fconvf,\ \text{for every}\  \vartheta \in \SOn.\]
     Lastly, the map $\oZ$ is said to be homogeneous of degree $j$ if $\oZ(\lambda v) = \lambda^j \oZ(v)$ for all $v \in \fconvf$ and $\lambda > 0$. 
     
     Let 
    \begin{align*}
        D^n_j \coloneqq\{ \xi  \in C_b((0,\infty)) \ ;\ \underset{{r \rightarrow 0^+}}{\lim}r^{n-j}\xi(r) = 0 \ \text{and} \ \underset{{r \rightarrow 0^+}}{\lim}\int_{r}^{\infty} p^{n-j-1} \xi(p)\d p \ \text{exists and is finite} \}
    \end{align*}
     for $0 \leq j \leq n-1$ and $D^n_n \coloneqq\{ \xi  \in C_b((0,\infty)) \ ;\ \underset{{r \rightarrow 0^+}}{\lim}\xi(r) \ \text{exists and is finite}\} $.
     
   The functional counterparts of the intrinsic volumes were introduced by A. Colesanti, M. Ludwig, and F. Mussnig in \mbox{\cite{colesanti_ludwig_mussnig_5}}, and for  $v \in \fconvf \cap C^2(\R^n) $ and $\xi_i \in D^n_i$, they are given by
   \begin{equation}\label{decomposition_volumes_v_C2_hessian}
       \oV_{i,\xi_i}^*(v) = \int_{\R^n}\xi_i(|x|)[\Hess v(x)]_i\d x, \quad i \in \{0,\dots,n\}.
    \end{equation}
   For $\xi_i\in D_i^n$, the functional defined by \eqref{decomposition_volumes_v_C2_hessian} admits a unique continuous extension to all of $\fconvf$ where we consider spaces of convex functions together with the topology associated with epi-convergence. The measures $[\Hess v(x)]_i \d x$ are then replaced by the Hessian measures defined in (\ref{hessian_abs_cont_lebesgue}).

    In the following theorem, we recall the functional version of Hadwiger's Theorem \ref{hadwiger_theorem} on $\fconvf$. It was first established in \cite{colesanti_ludwig_mussnig_5}, and was also obtained by J. Knoerr in \cite{knoerr_singular}.
    
    \begin{theo}[\cite{colesanti_ludwig_mussnig_5}, Theorem 1.5]\label{Hadwiger_decom_volumes}
    A functional $\oZ: \fconvf \rightarrow \R$ is a continuous, dually epi-translation invariant, rotation invariant
    valuation if and only if there exist functions $\xi_0 \in D^n_0, \dots, \xi_n \in D^n_n$ such that 
    \begin{align*}
        \oZ(v) = \sum_{i = 0}^n\oV_{i,\xi_i}^*(v)
    \end{align*}
    for every $v \in \fconvf$.
    \end{theo}

    Let 
    \begin{align*}
        \fconvfz\coloneqq \{ v \in \fconvf : v \ \text{is of class} \ C^2 \ \text{around the origin}\}.
    \end{align*}
    Using Hessian measures, the representation given in (\ref{decomposition_volumes_v_C2_hessian}) also holds for $v \in \fconvfz$. That is,
    \begin{align}\label{representation_scalar_fconvfz}
        \oV_{i,\xi_i}^*(v)=   \int_{\R^n} \xi_i(|x|) \d\Phi^n_i(v;x)
    \end{align}
    for every $v \in \fconvfz$ where $\xi_i \in D^n_i$ for each $i \in \{0,\dots,n\}$. See \cite[Lemma 3.23]{colesanti_ludwig_mussnig_5}.

    Very recently, together with F. Mussnig, we introduced a new family of vector-valued valuations on $\fconvf$ and characterized them. These valuations are called \emph{functional Minkowski vectors}. They arise from applying a Steiner formula to a functional moment vector defined on $\fconvf$. See \cite{mouamine_mussnig_1, knoerr_ulivelli_preprint_2024}. 

    For $0 \leq j \leq n$, let 
    \begin{align*}
        T^n_j \coloneqq \{\xi  \in C_b((0,\infty)) \ ;\ \underset{{r \rightarrow 0^+}}{\lim}r^{n-j+1}\xi(r) = 0\}.
    \end{align*}
   For  $v \in \fconvf \cap C^2(\R^n) $ and $\xi_k \in T^n_k$, the functional Minkowski vectors are given by
   \begin{equation}\label{decomposition_vectors_v_C2_hessian}
           \ot_{k,\xi_k}^*(v) = \int_{\R^n} \xi_k(|x|)x [\Hess v(x)]_k\d x, \quad k \in \{1,\dots,n\}.
    \end{equation}
   For $\xi_k\in T_k^n$, the functional defined by \eqref{decomposition_vectors_v_C2_hessian} admits a unique continuous extension to all of $\fconvf$. In the theorem below, we recall the characterization result that was obtained in \cite{mouamine_mussnig_2}.
    \begin{theo}[\cite{mouamine_mussnig_2}, Theorem C]\label{Hadwiger_decom_vector}
        A functional  $\ot: \fconvf \rightarrow \R^n$ is a continuous, dually epi-translation invariant, rotation equivariant valuation if and only if there exist functions $\xi_1 \in T^n_1, \dots, \xi_n \in T^n_n$ such that 
        \begin{align*}
            \ot(v) = \sum_{k = 1}^n \ot_{k,\xi_k}^*(v)
        \end{align*}
        for every $v \in \fconvf$.
\end{theo}
    
    The Hessian measures allow us to extend the representation given in (\ref{decomposition_vectors_v_C2_hessian}) to $v \in \fconvfz$; see \cite[Lemma 4.11]{mouamine_mussnig_2}. This means that 
     \begin{align}\label{representation_vector_fconvfz}
       \ot_{k,\xi_k}^*(v)=  \int_{\R^n} \xi_k(|x|)x  \d\Phi^n_k(v;x)
    \end{align}
    for every $v \in \fconvfz$ where $\xi_k \in T^n_k$ for every $k \in \{1,\dots,n\}$.

    For $x \in \R^n$, we denote by $|x|_{n-1} := \sqrt{x_1^2 + \cdots + x_{n-1}^2}$. For $t,s \geq 0$, we set 
    \[v_t(x)\coloneqq \begin{cases}\label{function_v_t}
        0 & \text{if}\ |x| \leq t \\
        |x|-t & \text{if}\ |x| \geq t,
    \end{cases}\]
    and 
    \[w_s(x)\coloneqq \begin{cases}
     0 & \text{if} \ |x| \leq s \ \text{and} \ x_n\geq 0\\
     0 & \text{if}\  |x|_{n-1} \leq s \ \text{and} \ x_n < 0\\
     |x| - s & \text{if}\  |x| > s \ \text{and} \ x_n\geq 0\\
     |x|_{n-1} - s & \text{if}\  |x|_{n-1} > s \ \text{and} \ x_n < 0,
    \end{cases}\]
    for every $x \in \R^n$. We note that $v_t$ and $w_s$ are elements of $\fconvfz$ for every $t, s> 0$. 
    These two functions are key in establishing Theorem \ref{Hadwiger_decom_volumes} and Theorem \ref{Hadwiger_decom_vector}, respectively. We present some results that are essential to proving Theorem \ref{main_theorem_monge_ampere}.
\begin{lem}[\cite{colesanti_ludwig_mussnig_5}, Lemma 2.15]\label{retrieving_densities_for_volumes}
        Let $1 \leq j \leq n-1$ and $\xi \in D^n_j$. Then, for every $t > 0$,
        \begin{align*}
            \oV_{j,\xi}^{*}(v_t)= \kappa_n \binom{n}{j} \left[t^{n-j}\xi(t) + (n-j)\int_{t}^{+\infty} r^{n-j-1}\xi(r)\d r\right]
        \end{align*}
           
    \end{lem}
 \begin{lem}[\cite{mouamine_mussnig_2}, Proposition 6.2]\label{retrieving_density_vector_valued}
    If $1 \leq j \leq n$ and $\xi \in T^n_j$, then for every $s > 0$,
    \begin{align*}
        \ot_{j,\xi}^{*}(w_s)=\frac{1}{n}\binom{n}{j}\kappa_{n-1}s^{n-j+1}\xi(s)\ e_n.
    \end{align*}
    \end{lem}

\subsection*{Integral transform}
 For $\xi \in C_b((0,\infty))$ and $s > 0 $, we define   
    \begin{align*}
        \cR\xi (s) \coloneqq s \xi(s) + \int_{s}^{\infty} \xi(r)\d r
    \end{align*}
    and we have $\cR\xi \in C_b((0,\infty))$. For $m \in \N$, we have (Cf. \cite[Lemma 3.6]{colesanti_ludwig_mussnig_6})
        \begin{align}\label{R_m}
            \cR^m \xi (s) = s^m\xi(s) + m \int_{s}^{\infty}r^{m-1}\xi(r)\d r,
        \end{align}
        for every $s>0$, where 
        \begin{align*}
            \cR^m \xi :=  \underbrace{\cR \circ \cdots \circ \cR}_m \xi.
        \end{align*}
    By integration by parts (Cf. \cite[Lemma 3.8]{colesanti_ludwig_mussnig_6}) and its proofs), the inverse operator to $\cR^l$ is given by 
    \[\cR^{-l}\rho (s) = (\cR^{-1})^l\rho(s) = \frac{\rho(s)}{s^l} - l \int_{s}^{\infty} \frac{\rho(t)}{t^{l+1}}\d t \]
    for $s>0$, where $\rho \in C_b((0,\infty))$.
    
     The following lemma on the integral transform $\cR$ will be useful in the proof of Theorem \ref{main_theorem_monge_ampere}.
    \begin{lem}[\cite{mouamine_mussnig_2}, Lemma 4.2]\label{R_bijection_T_n}
         For $0 \leq k \leq n$ and $0 \leq l \leq n-k$, the map $\cR^l : T^n_k \rightarrow T^{n-l}_k$ is a bijection.
    \end{lem}
    Suppose that $\gamma : (0,\infty) \times (0,\infty) \rightarrow \R$, then we denote by the transform $\cR_i$ the integral transform $\cR$ on the $i$-th entry, where $i \in \{1,2\}$, that is, 
    \begin{align*}
        \left(\cR_1\gamma\right)(s,t) = s \gamma(s,t) + \int_{s}^{\infty} \gamma(r,t)\d r
    \end{align*}
    for every $t> 0$. Similarly, we define $\cR_2$.
    
\subsection*{Mixed Monge-Ampère measures} We define the \emph{Monge--Amp\`ere measure} of $v \in \fconvf$, a Radon measure on $\R^n$, as
    \begin{align*}
        \MA(v;A) = V_n\left(\underset{x \in A}{\cup} \partial v(x)\right)
    \end{align*}
    for every Borel set $A \subset \R^n$. The \emph{mixed Monge--Amp\`ere  measure} which is associated to an $n$-tuple of functions in $\fconvf$, is defined as 
    \begin{align}\label{polarization_monge_ampere}
        \MA(\lambda_1 v_1 + \dots + \lambda_m v_m; \cdot ) = \sum_{i_1,\dots,i_n = 1}^{m}\lambda_{i_1}\cdots\lambda_{i_n}\MA(v_{i_1},\dots,v_{i_n}; \cdot)
    \end{align}
    where $m \in \N$, $v_1, \dots, v_m \in \fconvf$ and $\lambda_1, \dots, \lambda_m \geq 0$. Equation (\ref{polarization_monge_ampere}) uniquely determines the mixed Monge–Amp\`ere measure if we additionally assume that it is symmetric in its entries; see \cite{trudinger_wang_hessian_ii}. 

        The following proposition gives a description of the measures $\MA_0$ defined in (\ref{monge_ampere_j}).
    \begin{prop}[\cite{colesanti_ludwig_mussnig_7}, Theorem 4.5]\label{MA_0}
        Let $v \in \fconvf$. We have
        \begin{align*}
            \MA_0(v; B) = \kappa_n \delta_0(B)
        \end{align*}
        for every Borel set $B \subset \R^n$, where $\delta_0$ is the Dirac measure at the origin.
    \end{prop}
     We point out that the functionals defined in Theorem \ref{Hadwiger_decom_volumes} and Theorem \ref{Hadwiger_decom_vector} can be expressed in terms of the $\MA_j$ measures as follows.

\begin{theo}[\cite{colesanti_ludwig_mussnig_7}, Theorem 1.6]\label{functional_volume_decomposition}
    If $0 \leq i \leq n$ and $\xi_i \in D^n_i$, then 
    \begin{align}\label{decomposition_volumes_v_monge_ampere}
        \oV_{i,\xi_i}^*(v) =\int_{\R^n} \alpha_i(|x|) \d\MA_i(v;x)
    \end{align}
    for every $v \in \fconvf$, where $\alpha_i \in D^n_n$ and is given by
     \begin{align*}
        \alpha_i(p) = \binom{n}{i}\cR^{n-i}\xi_i(p),
    \end{align*}
    for every $p >0$, where the integral transform $\cR$ is defined in (\ref{R_m}).
\end{theo}
\begin{theo}[\cite{mouamine_mussnig_2}, Theorem B]\label{functional_vector_decomposition}
        Let $1 \leq k \leq n$ and $\xi_k \in T^n_k$, then
        \begin{align}\label{decomposition_vectors_monge_ampere}
            \ot_{k,\xi_k}^*(v) =  \int_{\R^n} \alpha_k(|x|)x \d\MA_k(v;x)
        \end{align}
        for every $v \in \fconvf$, where $\alpha_k \in T^n_n$ and is given by
            \begin{align*}
                \alpha_k(p) = \binom{n}{k}\cR^{n-k}\xi_k(p)
             \end{align*}
    for every $p >0$.
       
\end{theo}
    
    The relevant properties of the operators appearing in Theorem \ref{functional_volume_decomposition} and Theorem \ref{functional_vector_decomposition} are consequences of the following statements, which can be found in \cite[Proposition 5.3]{colesanti_ludwig_mussnig_7}.
    \begin{lem}\label{CVA}
    Let $\psi \in C_c(\R^n)$ and $0 \leq j \leq n$. If $v_1,\dots,v_{n-j} \in \fconvf$, then 
    \begin{align*}
        v \rightarrow \int_{\R^n} \psi(x) \d\MA(v[j], v_1,\dots,v_{n-j};x)
    \end{align*}    
    defines a continuous, dually epi-translation invariant valuation on $\fconvf$ that is homogeneous of degree $j$. 
\end{lem}
   
\section{Proof of Theorem \ref{main_theorem_monge_ampere}}\label{proof_main_theorem}
 Let $1\leq j\leq n$, and $\xi\in T_j^n$. Define 
    \begin{align}\label{the_expression}
        \oz_{j,\xi}(u,v)\coloneqq\int_{\SOn} \ot_{j,\xi}^*(u+(v\circ \vartheta^{-1})) \d \vartheta
    \end{align}
    for every $u, v \in \fconvf$ where the functional $\ot_{j,\xi}^*$ is defined in (\ref{decomposition_vectors_monge_ampere}). Since $ \ot_{j,\xi}^*$ is homogeneous of degree $j$, the map $ \oz_{j,\xi}$ is jointly homogeneous of degree $j$, that is,
    \[\oz_{j,\xi}(\tau u,\tau v) = \tau^j\oz_{j,\xi}( u, v),\]
    for all $\tau > 0$.
   
   Since $\xi \in T^n_j$, the functional $\ot_{j,\xi}^*$ is continuous by Theorem \ref{Hadwiger_decom_vector}, and since $v \mapsto v \circ \vartheta^{-1}$ is jointly continuous by Lemma \ref{JCM}, the map 
    \[\vartheta \mapsto \ot_{j,\xi}^*(u+(v\circ \vartheta^{-1}))\]
    is continuous on the compact group $\SOn$. Hence the integral in (\ref{the_expression}) is well defined.
    \begin{lem}\label{kin_form}
    Let $1\leq j\leq n$ and $\xi\in T_j^n$.
   
    \begin{enumerate}
        \item For a fixed $\Bar{v} \in \fconvf$, the map $\oz_{j,\xi}(\cdot,\bar{v}): \fconvf \rightarrow \R^n$ is a continuous, dually epi-translation invariant and rotation equivariant valuation.
        \item For a fixed $\Bar{u} \in \fconvf$, the map $\oz_{j,\xi}(\bar{u},\cdot): \fconvf \rightarrow \R^n$ is a continuous, dually epi-translation invariant and rotation invariant valuation.
    \end{enumerate}
\end{lem}
\begin{proof} 

Let $1 \leq j \leq n$, and $\xi \in T^n_j$. For a fixed $ \bar{v} \in \fconvf$, let
   \begin{align}\label{kinematic-form1}
         \oz_{j,\xi}(u,\bar{v}) =\int_{\SOn} \ot_{j,\xi}^*(u+(\bar{v}\circ \vartheta^{-1})) \d \vartheta
    \end{align}
    for every $u \in \fconvf$.
    \begin{enumerate}
        \item We begin by showing that the map 
        \begin{equation}\label{joint_cont_map1}
            (\vartheta,u) \rightarrow \ot_{j,\xi}^*(u+(\bar{v}\circ \vartheta^{-1}))
        \end{equation}
        is jointly continuous on $\SOn \times \fconvf$. Since the functional $\ot_{j,\xi}^*$ is rotation equivariant, 
        \begin{align*}
            \ot_{j,\xi}^*(u+(\bar{v}\circ \vartheta^{-1}))&= \ot_{j,\xi}^*(((u\circ \vartheta)+\bar{v})\circ \vartheta^{-1})\\
            &= \vartheta \ot_{j,\xi}^*((u\circ \vartheta)+\bar{v})
        \end{align*}
        for every $u \in \fconvf$ and $\vartheta \in \SOn$. Thus, by the continuity of addition under epi-convergence in $\fconvf$, Lemma \ref{JCM}, and the continuity of $ \ot_{j,\xi}^*$ , the map (\ref{joint_cont_map1}) is jointly continuous on  $\SOn \times \fconvf$.

        Now, let $u_i \in \fconvf$ with $i \in \N$ be such that $u_i$ epi-converges to some $\Bar{u} \in \fconvf$, meaning that $\{u_i \ : i \in \N\} \cup \{\Bar{u}\}$ is sequentially compact. And since $\SOn$ is compact and (\ref{joint_cont_map1}) is jointly continuous, it follows that the supremum 
        \begin{align*}
            \sup\Bigl\{\left| \ot_{j,\xi}^*(u_i+(\bar{v}\circ \vartheta^{-1}))\right|: i \in \N, \vartheta \in \SOn\Bigl\}
        \end{align*}
        is finite. Thus, by applying the dominated convergence theorem, we obtain
        \begin{align*}
            \lim_{i \rightarrow \infty} \int_{\SOn} \ot_{j,\xi}^*(u_i+(\bar{v}\circ \vartheta^{-1}))\d\vartheta= \int_{\SOn}  \ot_{j,\xi}^*(\bar{u}+(\bar{v}\circ \vartheta^{-1})) \d\vartheta
        \end{align*}
        which proves the continuity of
        \[u \mapsto  \oz_{j,\xi}(u,\bar{v}) .\]
        Similarly, for a fixed $\bar{u}\in \fconvf$, the map 
         \begin{align}\label{joint_cont_map2}
            (\vartheta,v) \rightarrow\ot_{j,\xi}^*(\Bar{u} + (v\circ \vartheta^{-1}) )
        \end{align}
        is jointly continuous on $\SOn \times \fconvf$. The same compactness and dominated convergence argument shows that 
      \begin{align*}
         v \mapsto \oz_{j,\xi}(\bar{u},v)
     \end{align*}
    is continuous.
        \item We now prove the dually epi-translation invariance. Let $y \in \R^n$, and let $l_y(x) = \langle x,y\rangle $ for every $x \in \R^n$. For a fixed $\bar{v} \in \fconvf$, we have
        \begin{align*}
            (u + l_y) + (\bar{v}\circ \vartheta^{-1}) = (u + (\bar{v}\circ \vartheta^{-1}))+l_y.
        \end{align*}
       Since the functional $\ot_{j,\xi}^*$ is dually epi-translation invariant, it follows that
       \[\oz_{j,\xi}(u + l_y,\bar{v}) = \oz_{j,\xi}(u,\bar{v}).\]
       For a fixed $\bar{u}\in\fconvf$, we use
       \begin{align*}
           (v + l_y) \circ \vartheta^{-1} = v\circ\vartheta^{-1} + l_{\vartheta y }.
       \end{align*}
       Hence,
         \begin{align*}
            \Bar{u} + ((v + l_y) \circ \vartheta^{-1} )= \Bar{u}+  (v \circ \vartheta^{-1}) +l_{\vartheta y }.
        \end{align*}
        Again, since the functional $\ot_{j,\xi}^*$ is dually epi-translation invariant, we obtain
       \[\oz_{j,\xi}(\bar{u},v + l_y) = \oz_{j,\xi}(\bar{u},v).\] 
        \item We then prove the valuation property. Let $u_1, u_2 \in \fconvf$, such that $u_1 \vee u_2$ and $u_1 \wedge u_2$ belong to $\fconvf$. Since
        \[(u_1 \wedge u_2) + (\bar{v}\circ \vartheta^{-1}) = (u_1 + (\bar{v}\circ \vartheta^{-1})) \wedge (u_2 + (\bar{v}\circ \vartheta^{-1})) \]
         and similarly for the maximum, the valuation property of $\ot_{j,\xi}^*$ implies that, 
         \[u \mapsto  \oz_{j,\xi}(u,\bar{v})\]
         is a valuation.

         The proof in the second variable is the same. Indeed, if $v_1, v_2 \in \fconvf$, such that $v_1 \vee v_2$ and $v_1 \wedge v_2$ belong to $\fconvf$, then 
         \begin{align}
             \bar{u}+ ((v_1 \wedge v_2)\circ \vartheta^{-1}) = (\bar{u}+ (v_1\circ \vartheta^{-1})) \wedge (\bar{u}+ (v_2\circ \vartheta^{-1})),
         \end{align}
         and similarly for the maximum. Hence, 
         \[v \mapsto \oz_{j,\xi}(\bar{u},v)\]
         is a valuation.
        \item It remains to show the rotation properties, where the two variables behave differently.
         \begin{enumerate}
             \item First, let $\bar{v}\in \fconvf$ be fixed and let $\phi \in \SOn$. Then,
        \begin{align*}
             \oz_{j,\xi}(u\circ \phi^{-1},\bar{v})&= \int_{\SOn} \ot_{j,\xi}^*((u \circ \phi^{-1})+(\bar{v}\circ \vartheta^{-1})) \d\vartheta 
            \\&=\int_{\SOn}\ot_{j,\xi}^*((u+ (\bar{v}\circ \vartheta^{-1}\circ\phi))\circ \phi^{-1}) \d\vartheta\\
            &= \phi \int_{\SOn} \ot_{j,\xi}^*(u + (\bar{v}\circ \psi^{-1})) \d\psi\\
            &= \phi \oz_{j,\xi}(u,\bar{v})
        \end{align*}
        for all $u \in \fconvf$, where we have used the rotation equivariance of the functional $\ot_{j,\xi}^*$, the change of variables, and the invariance of the Haar measure on $\SOn$. This proves the rotation equivariance of 
        \[u \mapsto \oz_{j,\xi}(u,\bar{v}).\]
        \item Second, let $\bar{u}\in \fconvf$ be fixed. Since
        \begin{align}\label{equivalent_kinematic_form2}
            \int_{\SOn}\ot_{j,\xi}^*(\Bar{u}+ v \circ \vartheta^{-1} )\d \vartheta= \int_{\SOn}\vartheta \ot_{j,\xi}^*((\Bar{u}\circ\vartheta)+ v)\d\vartheta,
        \end{align}
        then for every $\phi \in \SOn$, we have 
        \begin{align*}
             \oz_{j,\xi}(\bar{u},v\circ \phi^{-1})&=\int_{\SOn}\vartheta \ot_{j,\xi}^*((\Bar{u}\circ\vartheta)+ (v \circ \phi^{-1}))\d\vartheta \\
            &=\int_{\SOn} \vartheta \ot_{j,\xi}^*(((\Bar{u}\circ\vartheta\circ\phi)+v)\circ \phi^{-1} ) \d\vartheta \\
            &=\int_{\SOn} \vartheta\phi \ot_{j,\xi}^*((\Bar{u}\circ\vartheta\circ\phi)+ v) \d\vartheta \\
            &= \int_{\SOn}\psi \ot_{j,\xi}^*((\Bar{u}\circ\psi)+ v)\d\psi\\
            &= \oz_{j,\xi}(\bar{u},v)
        \end{align*}
        for all $v \in \fconvf$. We have used the rotation equivariance of the functional $\ot_{j,\xi}^*$ in the third equality, the change of variables, and the invariance of the Haar measure on $\SOn$ in the fourth equality.
        Hence, the functional
        \[v \mapsto \oz_{j,\xi}(\bar{u},v)\]
        is rotation invariant.
         \end{enumerate}
    \end{enumerate}
    Combining the preceding properties proves both assertions.
    \end{proof}
To prove Theorem \ref{main_theorem_monge_ampere}, we will use some lengthy calculations (Proposition \ref{retrieving_densities_composition}), which can be found in the Appendix \ref{Appendix}. 

\begin{proof}[Proof of Theorem \ref{main_theorem_monge_ampere}]
    If $j = 0$, the statement follows immediately. Let $1\leq j\leq n$, and let $\alpha\in T_n^n$. By Lemma \ref{R_bijection_T_n}, there exists $\xi \in T^n_j$ such that
    \begin{align*}
        \alpha(r) = \binom{n}{j}\cR^{n-j}\xi(r),
    \end{align*}
    for every $r> 0$.
    For $u, v \in \fconvf$, by Theorem \ref{functional_vector_decomposition}, we have
    \begin{equation}
    \begin{aligned}
        \oz_{j,\xi}(u,v) &=  \int_{\SOn} \ot_{j,\xi}^*(u+ (v\circ \vartheta^{-1})) \d\vartheta\\
        &= \int_{\SOn}\int_{\R^n}\alpha(|x|)x \d\MA_j(u+ (v\circ \vartheta^{-1});x)\d\vartheta. 
    \end{aligned}
    \end{equation}
     For a fixed $\Bar{v} \in \fconvf$, and according to Lemma \ref{kin_form}, the map $u \mapsto \oz_{j,\xi}(u,\Bar{v})$ is a continuous, dually epi-translation invariant, and rotation equivariant valuation on $\fconvf$. By Theorem \ref{Hadwiger_decom_vector}, and Theorem \ref{functional_vector_decomposition}, there exists $\xi^j_{k,\Bar{v}} \in T^n_k$ where $1 \leq k \leq n$, such that 
    \begin{align}\label{valuation_monge_ampere}
        \oz_{j,\xi}(u,\Bar{v}) = \sum_{k=1}^n \int_{\R^n} \alpha^j_{k,\Bar{v}}(|x|) x \d\MA_k(u;x)
    \end{align}
    for every $u \in \fconvf$, where 
    \begin{align*}
        \alpha^j_{k,\Bar{v}}(p) = \binom{n}{k}\cR^{n-k}\xi^j_{k,\Bar{v}}(p),
    \end{align*}
    for every $p >0$. We recall that, for every $x \in \R^n$,
    \[w_s(x)= \begin{cases}\label{function_w_s}
     0 & \text{if} \ |x| \leq s \ \text{and} \ x_n\geq 0\\
     0 & \text{if}\  |x|_{n-1} \leq s \ \text{and} \ x_n < 0\\
     |x| - s & \text{if}\  |x| > s \ \text{and} \ x_n\geq 0\\
     |x|_{n-1} - s & \text{if}\  |x|_{n-1} > s \ \text{and} \ x_n < 0
    \end{cases}.\]
    Since $w_s\in \fconvfz$, we may evaluate (\ref{valuation_monge_ampere}) at $ u = \mu w_s$. Using the $k$-homogeneity of the measure $\MA_k$, that is 
    \[\MA_k(\mu w_s; \cdot) = \mu^k \MA_k(w_s; \cdot),\]
    and then applying Lemma \ref{retrieving_density_vector_valued} to each summand with $\xi = \xi^j_{k,\Bar{v}}$, we obtain
    \begin{align*}
        \oz_{j,\xi}(\mu w_s, \Bar{v}) = \frac{1}{n}\kappa_{n-1}\sum_{k=1}^{n}\mu^k\binom{n}{k}s^{n-k+1}\xi^j_{k,\Bar{v}}(s)\ e_n
    \end{align*}
    for every $s > 0$ and every $\mu > 0$.
    
    Let us now fix $\Bar{s}> 0$. By Lemma \ref{kin_form} and for each $1 \leq k \leq n$, the map
    \begin{align*}
        v \mapsto \xi^j_{k,v}(\Bar{s})
    \end{align*}
    is a continuous, dually epi-translation invariant, and rotation invariant valuation. By Theorem \ref{Hadwiger_decom_volumes} and Theorem \ref{functional_volume_decomposition}, for each fixed $1 \leq k \leq n$, there exist functions $\xi^j_{k,i}\in D^n_i$, for $0 \leq i \leq n $, such that 
    \begin{align*}
        \xi^j_{k,v}(\Bar{s}) = \sum_{i=0}^n\int_{\R^n}\alpha^j_{k,i}(\Bar{s},|y|)\d\MA_i(v;y)
    \end{align*}
    for every $v \in \fconvf$, where
     \begin{align*}
        \alpha^j_{k,i}(\bar{s},q) = \binom{n}{i}\cR^{n-i}_2\xi^j_{k,i}(\bar{s},q), \quad q> 0.
    \end{align*}
   Since $\bar{s} > 0$ was arbitrary, this defines $\alpha^j_{k,i}$ on $(0,\infty)\times (0,\infty)$. Substituting the preceding expression into (\ref{valuation_monge_ampere}) gives,
    \begin{align*}
        \oz_{j,\xi}(u,v) = \sum_{i=0}^n\sum_{k=1}^n \int_{\R^n}\int_{\R^n}\alpha^j_{k,i}(|x|,|y|)x\d\MA_i(v;y)\d\MA_k(u;x)
    \end{align*}
    for every $u, v \in \fconvf$. Since $\ot_{j,\xi}^*$ is homogeneous of degree $j$, the map $ \oz_{j,\xi}$ is jointly homogeneous of degree $j$ in the argument $ u + (v\circ \vartheta^{-1}) $. This implies that the only terms satisfying $k+i = j$ can occur. Therefore,
    \begin{align}\label{double_sum_exp}
         \oz_{j,\xi}(u,v) = \sum_{k=1}^{j} \int_{\R^n}\int_{\R^n}\alpha^j_{k,j-k}(|x|,|y|)x\d\MA_{j-k}(v;y)\d\MA_k(u;x)
    \end{align}
    for every $u, v \in \fconvf$, where
    \begin{align}\label{alpha_j_k}
        \alpha^j_{k,j-k}(p,q) = \binom{n}{k}\binom{n}{j-k}\cR^{n-k}_1\left(\cR^{n-j+k}_2\left(\xi^j_{k,j-k}\right)\right)(p,q)
    \end{align}
    for every $p,q >0$.
    For simplicity, we will denote $\alpha^j_{k,j-k}$ by $\alpha_{k,j-k}$.
    
   It remains to determine the functions $\alpha_{k,j-k}$ in terms of $\alpha$. We evaluate (\ref{double_sum_exp}) at $u = \mu w_s$ and $v=\lambda v_t$, where $s,t > 0$ and $ \mu, \lambda > 0$. For $1 \leq k \leq j-1$, the $k$-th term in (\ref{double_sum_exp}) is
    \[\int_{\R^n}\int_{\R^n}\alpha_{k,j-k}(|x|,|y|)x\,\d\MA_{j-k}(\lambda v_t;y)\d \MA_k(\mu w_s;x).
    \]
   By the homogeneity of the mixed Monge--Amp\`ere measures, this is equal to
    \[
       \mu^k\lambda^{j-k}\int_{\mathbb R^n}\int_{\mathbb R^n}\alpha_{k,j-k}(|x|,|y|)x\, \d\MA_{j-k}(v_t;y)\d\MA_k(w_s;x).
    \]
    For every fixed $x \in \R^n$, applying Lemma \ref{retrieving_densities_for_volumes} together with the representation (\ref{decomposition_volumes_v_monge_ampere}) to the function
    \[ q\mapsto \alpha_{k,j-k}(|x|,q),\]
   this gives,    
   \[\int_{\R^n}\alpha_{k,j-k}(|x|,|y|)\d\MA_{j-k}(v_t;y)=\kappa_n\,\alpha_{k,j-k}(|x|,t).
    \]
    Hence the $k$-th mixed term becomes
    \[\mu^k\lambda^{j-k}\kappa_n\int_{\R^n}\alpha_{k,j-k}(|x|,t)x\,\d\MA_k(w_s;x).
    \]
    Applying Lemma \ref{retrieving_density_vector_valued} together with the representation (\ref{decomposition_vectors_monge_ampere}) to the function
    \[p\mapsto \alpha_{k,j-k}(p,t),\]
    yields
    \[\int_{\R^n}\alpha_{k,j-k}(|x|,t)x\,\d\MA_k(w_s;x)=\frac1n\kappa_{n-1}s^{n-k+1}\cR_1^{-(n-k)}\alpha_{k,j-k}(s,t)e_n.
    \]
    Thus the contribution of the $k$-th mixed term is
    \[\frac1n\kappa_{n-1}\kappa_n\mu^k\lambda^{j-k}s^{n-k+1}\cR_1^{-(n-k)}\alpha_{k,j-k}(s,t)e_n.
    \]
    It remains to treat the endpoint term $k=j$, that is, the term in (\ref{double_sum_exp}) involving $\MA_0(\lambda v_t;\cdot)$. By Proposition \ref{MA_0},
    \[\MA_0(\lambda v_t;\cdot)=\kappa_n\delta_0.\]
     Therefore this contribution is given by
    \[\mu^j\kappa_n\int_{\R^n}\alpha_{j,0}(|x|,0)x\,\d\MA_j(w_s;x).\]
   Applying Lemma \ref{retrieving_density_vector_valued} once more together with the representation (\ref{decomposition_vectors_monge_ampere}) gives
    \[\mu^j\kappa_n\int_{\R^n}\alpha_{j,0}(|x|,0)x\,\d\MA_j(w_s;x)=\frac1n\kappa_{n-1}\kappa_n\mu^j s^{n-j+1}\cR_1^{-(n-j)}\alpha_{j,0}(s,0)e_n.\]
    Combining the endpoint term with the mixed terms $1 \leq k \leq j-1$, we obtain
    \begin{align}\label{for_j_between_1_and_n}
        \oz_{j,\xi}( \mu  w_s, \lambda v_t)&=\frac{1}{n}\kappa_{n-1}\kappa_n \left[\mu^j s^{n-j+1}\cR^{-n+j}_1\alpha_{j,0}(s,0)+  \sum_{k=1}^{j-1}\mu^{k}\lambda^{j-k}s^{n-k+1} \cR^{-n+k}_1\alpha_{k,j-k}(s,t)\right] e_n,
    \end{align}
    for every $1 \leq j \leq n$. On the other hand, by Proposition \ref{retrieving_densities_composition}, we have 
    \begin{align}\label{for_j_between_1_and_n_other}
        \oz_{j,\xi}( \mu  w_s, \lambda v_t) =\frac{1}{n}\kappa_{n-1}\left[\mu^j
        s^{n-j+1}\cR^{-n+j}\alpha(s) +\sum_{k=1}^{j-1}\binom{j}{k} \lambda^{j-k}\mu^k s^{n-k+1} \cR^{-n+k}\alpha\left(\max\{s,t\}\right)\right] \ e_n
    \end{align}
    for every $1 \leq j \leq n$ and every $s,t > 0$. Since $\lambda, \mu > 0$ are arbitrary, we compare the coefficients of the monomials in $\lambda$ and $\mu$. The coefficient of the endpoint term $\mu^j$, that is, the term in (\ref{double_sum_exp}) corresponding to $k = j$ and $\MA_0(v)$, gives
    \begin{align*}
        \kappa_n\cR^{-n+j}_1\alpha_{j,0}(s,0) =   \cR^{-n+j}\alpha(s).
    \end{align*}
    For $1 \leq k \leq j-1$, comparison of the coefficient of $\mu^k \lambda^{j-k}$ gives
    \begin{align*}
        \kappa_n \cR^{-n+k}_1\left(\alpha_{k,j-k}\right)(s,t) = \binom{j}{k}\left(\cR^{-n+k}\alpha\right)\left(\max\{s,t\}\right)
    \end{align*}
    Applying $\cR_1^{\,n-j} $to the endpoint identity gives
    \[\alpha_{j,0}(s,0)=\frac1{\kappa_n}\alpha(s).\]

    For $1\leq k\leq j-1$, applying $\cR_1^{\,n-k}$ in the first variable gives
    \[\alpha_{k,j-k}(s,t)=\frac1{\kappa_n}\binom jk\left[\cR_1^{\,n-k}\left((\cR^{-n+k}\alpha)(\max\{\cdot,\cdot\})\right)\right](s,t).
    \]
   We now use the following identity: for every $\beta \in C_b((0,\infty))$ and every $m\in\mathbb N_0$,
    
    \[\mathcal R_1^m\big(\beta(\max\{\cdot,\cdot\})\big)(s,t)=(\mathcal R^m\beta)(\max\{s,t\}).
    \]
   Indeed, for \(m=1\), if \(s\geq t\), then
    \[\cR_1\big(\beta(\max\{\cdot,\cdot\})\big)(s,t)=s\beta(s)+\int_s^\infty \beta(r)\d r=(\cR\beta)(s).\]
   If \(s<t\), then
    \[\begin{aligned}
    \cR_1\big(\beta(\max\{\cdot,\cdot\})\big)(s,t)&=s\beta(t)+\int_s^t\beta(t)\d r+\int_t^\infty\beta(r)\d r \\
    &=t\beta(t)+\int_t^\infty\beta(r)\d r \\
    &=(\cR\beta)(t).
    \end{aligned}
    \]
   Thus the identity holds for \(m=1\), and the general case follows by iteration.
    Applying this identity with
    \[m=n-k,\qquad \beta=\cR^{-n+k}\alpha,\]
   we obtain
    \[\left[\cR_1^{\,n-k}\left((\cR^{-n+k}\alpha)(\max\{\cdot,\cdot\})\right)\right](s,t)=\alpha(\max\{s,t\}).
    \]
   Consequently,
    \[\alpha_{k,j-k}(s,t)=\frac1{\kappa_n}\binom jk\alpha(\max\{s,t\})
    \]
   for every $1\leq k\leq j-1$.
    
    The endpoint is consistent with the same expression, since
    \[\alpha_{j,0}(s,0)=\frac1{\kappa_n}\alpha(s)=\frac1{\kappa_n}\binom jj\alpha(\max\{s,0\}).
    \]
    Substituting these identities into (\ref{double_sum_exp}) gives the desired formula.
    \end{proof}
    To obtain the result in Corollary \ref{substituting_expression}, we need the following lemma.  
    \begin{lem}[\cite{hug_mussnig_ulivelli_support}, Lemma 4.6]\label{monge_ampere_for_support}
        If $K \in \Kn$, then
        \begin{align*}
            \binom{n}{j}\MA(h_{K}[j],h_{B^n}[n-j];B) = \kappa_{n-j}V_j(K)\ \delta_0(B)
        \end{align*}
        for every Borel set $B \subset \R^n$ and $0 \leq j \leq n$.
    \end{lem}
    \begin{proof}[Proof of Corollary \ref{substituting_expression}]
        Since $\alpha(\max\{|x|,0\}) = \alpha(|x|)$ for all $x \in \R^n$, the proof follows from Theorem \ref{main_theorem_monge_ampere}, and by applying Lemma \ref{monge_ampere_for_support}. 
    \end{proof}
    \subsection*{Acknowledgments}
    The author thanks F. Mussnig and M. Ludwig for fruitful discussions and remarks. This project was supported by the Austrian Science Fund (FWF) 10.55776/P36210 and by PRIMUS/24/SCI/009 provided by Charles University.

\newpage

\appendix
\section{Appendix}\label{Appendix}
As the proof of Proposition \ref{retrieving_densities_composition} is long and technical, we decided to put it in the appendix. 

We recall that for $t,s \geq 0$, 
    \[v_t(x)= \begin{cases}
        0 & \text{if}\ |x| \leq t \\
        |x|-t & \text{if}\ |x| \geq t
    \end{cases}\]
    and 
    \[w_s(x)= \begin{cases}
     0 & \text{if} \ |x| \leq s \ \text{and} \ x_n\geq 0\\
     0 & \text{if}\  |x|_{n-1} \leq s \ \text{and} \ x_n < 0\\
     |x| - s & \text{if}\  |x| > s \ \text{and} \ x_n\geq 0\\
     |x|_{n-1} - s & \text{if}\  |x|_{n-1} > s \ \text{and} \ x_n < 0
    \end{cases}\]
    for every $x \in \R^n$ where $|x|_{n-1}= \sqrt{x_1^2+\dots+x_{n-1}^2}$. Then, the following lemma is easy to see. 
\begin{lem}\label{sum_of_two_functions}
    For $\mu, \lambda \geq 0$, we have the following:
\begin{enumerate}
    \item If $s \leq t$, then 
        $$(\mu \ w_s + \lambda \ v_t)(x)= \begin{cases}
    0 & \text{if} \ |x| \leq s \ \text{and} \ x_n\geq 0\\
    \mu(|x| - s) & \text{if}\  s < |x| \leq t \ \text{and} \ x_n \geq 0\\
   \mu(|x| - s) + \lambda(|x|-t) & \text{if}\  |x| > t \ \text{and} \ x_n\geq 0\\
    0 & \text{if}\  |x|_{n-1} \leq s, \ |x|\leq t \ \text{and} \ x_n < 0\\
    \lambda( |x| - t) & \text{if}\  |x|_{n-1} \leq s, \ |x|> t \ \text{and} \ x_n < 0\\ 
    \mu(|x|_{n-1} - s) & \text{if}\  |x|_{n-1} > s, \ |x| \leq t \ \text{and} \ x_n < 0\\ 
   \mu(|x|_{n-1} - s) + \lambda(|x| - t) & \text{if}\  |x|_{n-1} > s, \ |x| > t \ \text{and} \ x_n < 0. 
        \end{cases}$$
    \item If $s > t $, then 
     $$(\mu \ w_s + \lambda \ v_t)(x)= \begin{cases}
      0 & \text{if} \ |x| \leq t \ \text{and} \ x_n\geq 0\\
    \lambda(|x| - t) & \text{if}\  t < |x| \leq s \ \text{and} \ x_n \geq 0\\
   \mu(|x| - s) + \lambda(|x|-t) & \text{if}\  |x| > s \ \text{and} \ x_n\geq 0\\
    0 & \text{if}\   |x|\leq t \ \text{and} \ x_n < 0\\
    \lambda( |x| - t) & \text{if}\  |x|_{n-1} \leq s, \ |x|> t \ \text{and} \ x_n < 0\\
   \mu(|x|_{n-1} - s) + \lambda(|x| - t) & \text{if}\  |x|_{n-1} > s \ \text{and} \ x_n < 0. 
     \end{cases}$$
\end{enumerate}
\end{lem}\begin{figure}[H]
\centering

% =========================================================
% LEFT PANEL : s <= t
% =========================================================
\begin{subfigure}{0.48\textwidth}
\centering
\begin{tikzpicture}[scale=0.82, >=Latex, line cap=round, line join=round, font=\small]

% parameters
\def\R{3.8}
\def\s{1.35}
\def\t{2.35}
\pgfmathsetmacro{\yb}{sqrt(\t*\t-\s*\s)}
\pgfmathsetmacro{\ang}{acos(\s/\t)}

% -------------------------
% Region fills
% -------------------------
% A2
\begin{scope}
  \clip (-\R,0) rectangle (\R,\R);
  \fill[orange!22, even odd rule]
    (-\R,0) rectangle (\R,\R)
    (0,0) circle (\t);
\end{scope}

% A1
\begin{scope}
  \clip (-\R,0) rectangle (\R,\R);
  \fill[yellow!35, even odd rule]
    (0,0) circle (\t)
    (0,0) circle (\s);
\end{scope}

% A5
\begin{scope}
  \clip (-\R,-\R) rectangle (\R,0);
  \fill[violet!20, even odd rule]
    (-\R,-\R) rectangle (-\s,0)
    (0,0) circle (\t);
  \fill[violet!20, even odd rule]
    (\s,-\R) rectangle (\R,0)
    (0,0) circle (\t);
\end{scope}

% A4
\begin{scope}
  \clip (0,0) circle (\t);
  \clip (-\R,-\R) rectangle (\R,0);
  \fill[green!32] (-\R,-\R) rectangle (-\s,0);
  \fill[green!32] (\s,-\R) rectangle (\R,0);
\end{scope}

% A3
\begin{scope}
  \clip (-\R,-\R) rectangle (\R,0);
  \fill[cyan!25, even odd rule]
    (-\s,-\R) rectangle (\s,0)
    (0,0) circle (\t);
\end{scope}

% -------------------------
% Axes
% -------------------------
\draw[->, thick] (-\R-0.2,0) -- (\R+0.35,0) node[right] {$x_1$};
\draw[->, thick] (0,-\R-0.2) -- (0,\R+0.35) node[above] {$x_2$};
\fill (0,0) circle (1.1pt);
\node[below left] at (0,0) {$O$};

% ticks
\draw (\s,0.08) -- (\s,-0.08) node[below=4pt] {$s$};
\draw (-\s,0.08) -- (-\s,-0.08) node[below=4pt] {$-s$};
\draw (\t,0.08) -- (\t,-0.08) node[below=4pt] {$t$};
\draw (-\t,0.08) -- (-\t,-0.08) node[below=4pt] {$-t$};

\draw (0.08,\s) -- (-0.08,\s) node[left=4pt] {$s$};
\draw (0.08,\t) -- (-0.08,\t) node[left=4pt] {$t$};

% -------------------------
% Singular sets / boundaries
% -------------------------
% B1
\draw[red!80!black, very thick]
  (180:\s) arc[start angle=180,end angle=0,radius=\s];

% B2 : upper semicircle of radius t
\draw[magenta!80!black, very thick]
  (180:\t) arc[start angle=180,end angle=0,radius=\t];

% B3 : lower central arc of radius t, where |x_1| <= s
\draw[blue!80!black, very thick]
  ({180+\ang}:\t) arc[start angle={180+\ang},end angle={360-\ang},radius=\t];

% B6 : lower side arcs of radius t, where |x_1| > s
\draw[purple!85!black, very thick]
  (180:\t) arc[start angle=180,end angle={180+\ang},radius=\t];
\draw[purple!85!black, very thick]
  ({360-\ang}:\t) arc[start angle={360-\ang},end angle=360,radius=\t];

% B4 and B5 : vertical pieces x_1 = +/- s below the axis
\draw[teal!80!black, very thick] (-\s,0) -- (-\s,-\yb);
\draw[teal!80!black, very thick] (\s,0) -- (\s,-\yb);

\draw[brown!80!black, very thick] (-\s,-\yb) -- (-\s,-\R);
\draw[brown!80!black, very thick] (\s,-\yb) -- (\s,-\R);

% B7 points
\fill[black] (-\s,-\yb) circle (1.4pt);
\fill[black] (\s,-\yb) circle (1.4pt);

% -------------------------
% Region labels
% -------------------------
\node[draw=orange!70!black, fill=white, rounded corners=2pt]
  at (0,3.00) {$A_2$};

\node[draw=yellow!60!black, fill=white, rounded corners=2pt]
  at (0,1.80) {$A_1$};

\node[draw=cyan!60!black, fill=white, rounded corners=2pt]
  at (0.,-3.7) {$A_3$};

\node[draw=cyan!50!black, fill=white, rounded corners=2pt,inner sep=1pt] at (-1.78,-0.9) {$A_4$};
%\node[draw=green!50!black, fill=white, rounded corners=2pt]
  %at (1.95,-0.78) {$A_4$};

\node[draw=violet!60!black, fill=white, rounded corners=2pt]
  at (-2.95,-2.25) {$A_5$};
%\node[draw=violet!60!black, fill=white, rounded corners=2pt]
  %at (2.95,-2.25) {$A_5$};

% -------------------------
% Boundary labels
% -------------------------
\node[red!80!black, fill=white, inner sep=1pt] at (-0.65,0.95) {$B_1$};
\node[magenta!80!black, fill=white, inner sep=1pt] at (1.15,2.15) {$B_2$};
\node[blue!80!black, fill=white, inner sep=1pt] at (0.00,-2.30) {$B_3$};

%\node[teal!80!black, fill=white, inner sep=1pt] at (-1.72,-0.55) %{$B_4$};
\node[teal!80!black, fill=white, inner sep=1pt] at (1.5,-0.9) {$B_4$};

\node[brown!80!black, fill=white, inner sep=1pt] at (-1.5,-3.00) {$B_5$};
%\node[brown!80!black, fill=white, inner sep=1pt] at (1.72,-3.00) %{$B_5$};

%\node[purple!85!black, fill=white, inner sep=1pt] at (-2.45,-1.55) %{$B_6$};
\node[purple!85!black, fill=white, inner sep=1pt] at (2.,-1.5) {$B_6$};

\node[fill=white, inner sep=1pt] at (0,-2.95) {$B_7$};
\draw[->, thin] (-0.18,-2.80) -- (-\s,-\yb);
\draw[->, thin] (0.18,-2.80) -- (\s,-\yb);

\end{tikzpicture}
\caption{$s\le t$}
\end{subfigure}
\hfill
% =========================================================
% RIGHT PANEL : t < s
% =========================================================
\begin{subfigure}{0.48\textwidth}
\centering
\begin{tikzpicture}[scale=0.82, >=Latex, line cap=round, line join=round, font=\small]

% parameters
\def\R{3.8}
\def\t{1.35}
\def\s{2.35}

% -------------------------
% Region fills
% -------------------------
% C2
\begin{scope}
  \clip (-\R,0) rectangle (\R,\R);
  \fill[orange!22, even odd rule]
    (-\R,0) rectangle (\R,\R)
    (0,0) circle (\s);
\end{scope}

% C1
\begin{scope}
  \clip (-\R,0) rectangle (\R,\R);
  \fill[yellow!35, even odd rule]
    (0,0) circle (\s)
    (0,0) circle (\t);
\end{scope}

% C4
\begin{scope}
  \clip (-\R,-\R) rectangle (\R,0);
  \fill[violet!20] (-\R,-\R) rectangle (-\s,0);
  \fill[violet!20] (\s,-\R) rectangle (\R,0);
\end{scope}

% C3
\begin{scope}
  \clip (-\R,-\R) rectangle (\R,0);
  \fill[cyan!25, even odd rule]
    (-\s,-\R) rectangle (\s,0)
    (0,0) circle (\t);
\end{scope}

% -------------------------
% Axes
% -------------------------
\draw[->, thick] (-\R-0.2,0) -- (\R+0.35,0) node[right] {$x_1$};
\draw[->, thick] (0,-\R-0.2) -- (0,\R+0.35) node[above] {$x_2$};
\fill (0,0) circle (1.1pt);
\node[below left] at (0,0) {$O$};

% ticks
\draw (\t,0.08) -- (\t,-0.08) node[below=4pt] {$t$};
\draw (-\t,0.08) -- (-\t,-0.08) node[below=4pt] {$-t$};
\draw (\s,0.08) -- (\s,-0.08) node[below=4pt] {$s$};
\draw (-\s,0.08) -- (-\s,-0.08) node[below=4pt] {$-s$};

\draw (0.08,\t) -- (-0.08,\t) node[left=4pt] {$t$};
\draw (0.08,\s) -- (-0.08,\s) node[left=4pt] {$s$};

% -------------------------
% Singular sets / boundaries
% -------------------------
% D1 : upper semicircle of radius t
\draw[magenta!80!black, very thick]
  (180:\t) arc[start angle=180,end angle=0,radius=\t];

% D3 : lower semicircle of radius t
\draw[blue!80!black, very thick]
  (180:\t) arc[start angle=180,end angle=360,radius=\t];

% D2 : upper semicircle of radius s
\draw[red!80!black, very thick]
  (180:\s) arc[start angle=180,end angle=0,radius=\s];

% D4 : vertical lines below
\draw[teal!80!black, very thick]
  (-\s,0) -- (-\s,-\R)
  (\s,0) -- (\s,-\R);

% -------------------------
% Region labels
% -------------------------
\node[draw=orange!70!black, fill=white, rounded corners=2pt]
  at (0,3.00) {$C_2$};

\node[draw=yellow!60!black, fill=white, rounded corners=2pt]
  at (0,1.85) {$C_1$};

\node[draw=cyan!60!black, fill=white, rounded corners=2pt]
  at (0.,-2.45) {$C_3$};

\node[draw=violet!60!black, fill=white, rounded corners=2pt]
  at (-2.95,-2.20) {$C_4$};
%\node[draw=violet!60!black, fill=white, rounded corners=2pt]
 % at (2.95,-2.20) {$C_4$};

% -------------------------
% Boundary labels
% -------------------------
\node[magenta!80!black, fill=white, inner sep=1pt] at (0.95,1.05) {$D_1$};
\node[red!80!black, fill=white, inner sep=1pt] at (1.10,2.10) {$D_2$};
\node[blue!80!black, fill=white, inner sep=1pt] at (0.95,-1.05) {$D_3$};

%\node[teal!80!black, fill=white, inner sep=1pt] at (-2.72,-1.00) %{$D_4$};
\node[teal!80!black, fill=white, inner sep=1pt] at (2.5,-1.50) {$D_4$};

\end{tikzpicture}
\caption{$t<s$}
\end{subfigure}

\caption{Two-dimensional top view of the decomposition for \(\mu w_s+\lambda v_t\), with coordinates \(x=(x_1,x_2)\). In dimension \(2\), \(|x|=\sqrt{x_1^2+x_2^2}\) and \(|x|_{n-1}=|x_1|\). In the case \(s\le t\), the region \(A_5\) is the lower outer region \(\{|x_1|>s,\ |x|>t,\ x_2<0\}\).}
\label{fig:appendix-decomposition}
\end{figure}

 The following lemma provides a decomposition of Hessian measures of functions depending on less than $n$ variables.

    \begin{lem}[\cite{colesanti_ludwig_mussnig_8}, Lemma 2.9]\label{addition_complementary}
		Let $k\in\{1,\ldots,n-1\}$, $j\in\{0,\ldots,k\}$, and $v\in\fconvf$. If there exists $w\in\fconvfk$ such that
		\[
		v(x_1,\ldots,x_n)=w(x_1,\ldots,x_k)
		\]
		for every $(x_1,\ldots,x_n) \in\R^n$, then
		\[
		\d\Phi^n_j(v;(x_1,\ldots,x_n))=\d\Phi_j^{(k)}(w;(x_1,\ldots,x_k))\d x_{k+1}\cdots \d x_n.
		\]
	\end{lem}
    
To prove Proposition \ref{retrieving_densities_composition}, an essential tool will be the \emph{disintegration theorem}.
\begin{theo}[\cite{maggi_2023}, Theorem 17.1]\label{disintegration_theorem}
    Let $\mu$ be a finite Borel measure concentrated on a Borel set $E \subset \R^n$, and let $\pi: E \rightarrow \R$ be the projection map into the last coordinate. Then the following property holds: there exists a Borel set $\Sigma \subset \R$ of full $(\pi_{\#}\mu)$-measure such that for every $\sigma \in \Sigma$, there exists $\mu_{\sigma} \in \mathcal{P}(\R^n)$ such that $\mu_\sigma$ is concentrated on $\pi^{-1}(\sigma)$, and 
        \begin{align*}
            \int_{\R^n}\varphi \d\mu = \int_{\Sigma}\d(\pi_{\#}\mu)(\sigma)\int_{\R^n} \varphi \d\mu_{\sigma},
        \end{align*}
        for every bounded Borel function $\varphi: \R^n \rightarrow \R$, where $(\pi_{\#}\mu)$ denotes the push-forward of the measure $\mu$ through the map $\pi$, that is, 
        \begin{align*}
            (\pi_{\#}\mu)(B) = \mu(\pi^{-1}(B)),
        \end{align*}
        for all $B \subset \R$ Borel set.
\end{theo}
Evaluating the map given by (\ref{the_expression}) at $\mu w_s$ and $\lambda v_t$ where $s,t > 0$, and $\mu, \lambda > 0$, yields the following result.
\begin{prop}\label{retrieving_densities_composition}
    If $1 \leq j \leq n$ and $\xi \in T^n_j$, then
   \begin{align*}
     \oz_{j,\xi}(\mu w_s,\lambda v_t)=\frac{1}{n}\kappa_{n-1}\left[\mu^j
        s^{n-j+1}\cR^{-n+j}\alpha(s) +\sum_{k=1}^{j-1}\binom{j}{k} \lambda^{j-k}\mu^k s^{n-k+1} \cR^{-n+k}\alpha\left(\max\{s,t\}\right)\right] \ e_n
   \end{align*}
   for every $s, t > 0$ and every $\mu, \lambda > 0$, where 
   \begin{align*}
       \alpha(r)= \binom{n}{j}\cR^{n-j}\xi(r)
   \end{align*}
   for every $r>0$.
\end{prop}
\begin{proof}
    Let $1 \leq j \leq n$, $\xi \in T^n_j$, and $\mu, \lambda > 0$. Since the function $\mu w_s + \lambda v_t \in \fconvfz$, it follows by (\ref{representation_vector_fconvfz}) and by the rotation invariance of the function $v_t$, that
    \begin{equation}\label{the_expression_2}
        \begin{aligned}
            \oz_{j,\xi}(\mu w_s,\lambda v_t)&=\int_{\SOn} \ot_{j,\xi}^*(\mu w_s+(\lambda v_t\circ \vartheta^{-1})) \d \vartheta \\
          &= \int_{\R^n}\xi(|x|)x\d\Phi^n_j(\mu  w_s + \lambda v_t; x)
        \end{aligned}
    \end{equation}
    for all $s,t > 0$. By Lemma \ref{sum_of_two_functions}, we will need to distinguish between the cases when $s \leq t$ and $s > t$ in the last integral of (\ref{the_expression_2}). The hypersurfaces on which the expression of \(\mu w_s+\lambda v_t\) changes are
    \[|x|=s,\qquad |x|=t,\qquad |x|_{n-1}=s,\qquad x_n=0.
    \]
    Their relative position depends on whether \(s\leq t\) or \(t<s\). This is why the proof is divided into two cases. In each case, the open regions are those on which \(\mu w_s+\lambda v_t\) is \(C^2\), and their contributions are  computed using the density formula
    \[
    \d\Phi_j^n(v;x)=[\Hess v(x)]_j\,\d x.
    \]
    The remaining pieces are lower-dimensional transition regions contained in
    the hypersurfaces above. We call them singular regions only in this sense: they are not open \(C^2\)-regions, so the density formula is not applied to them directly. Their contributions are obtained either by local determination on an explicit open neighborhood, or by computing the corresponding parallel sets and, when necessary, by disintegrating the resulting Hessian measure.

    Let us first point out that for every $\vartheta \in \SOn$ such that $\vartheta e_n= e_n$, then $\vartheta$ acts only on the first $n-1$ coordinates and therefore preserves the quantities $|x|$, $|x|_{n-1}$, and $x_n$. Since both $w_s$ and $v_t$ depend only on these quantities, it follows directly from the definitions that 
    \begin{align*}
        v_t \circ \vartheta^{-1}=v_t \quad \text{and} \quad w_s \circ \vartheta^{-1}=w_s
    \end{align*}
    for every $t,s \geq 0$. Using the equivariance of $\oz_{j,\xi}$, one then obtains
   
    \begin{align*}
       \oz_{j,\xi}(\mu w_s, \lambda v_t  ) = \oz_{j,\xi}(\mu w_s \circ \vartheta^{-1}, \lambda v_t\circ \vartheta^{-1}) = \vartheta \oz_{j,\xi}(\mu w_s, \lambda v_t),
    \end{align*}
    hence $\oz_{j,\xi}(\mu w_s, \lambda v_t  )$ is invariant under all rotations fixing $e_n$ and therefore  $\oz_{j,\xi}(\mu w_s, \lambda v_t  )$ must be parallel to $e_n$. 
    
    \begin{enumerate}
        \item Let $s < t$. The decomposition of $\R^n$ consists of five $C^2$-regions $A_1,\dots,A_5$ and seven singular regions $B_1,\dots,B_7$. In this case the upper half-space is cut by the two spherical levels $|x|=s$ and $|x|=t$, whereas in the lower half-space the cylindrical level $|x|_{n-1}=s$ interacts with the spherical level $|x|=t$. The regions $A_1,\ldots,A_5$ are precisely the open pieces on which the formula in Lemma \ref{sum_of_two_functions} is smooth. The regions $B_1,\ldots,B_7$ are the transition pieces where one crosses one or more of the hypersurfaces $|x|=s$, $|x|=t$, $|x|_{n-1}=s$, and $x_n=0$. The computation below therefore separates the absolutely continuous contributions from the boundary contributions.
        \begin{enumerate}
            \item The function $\mu  w_s + \lambda v_t$ is of class $C^2$ on the set $A_1=\{x \in \R^{n} : s < |x| < t, x_n > 0\}$, and the Hessian matrix of $\mu  w_s + \lambda v_t$ at a point $x$ in this set has $(n-1)$-eigenvalues equal to $\mu/|x|$ and the last eigenvalue is equal to zero. Hence,
                $$[\Hess  (\mu  w_s + \lambda v_t)(x)]_j=\begin{cases}
                     \binom{n-1}{j}\frac{\mu^j}{|x|^j} & \text{if}\ 0 \leq j \leq n-1 \\
                    0 & \text{if}\ j=n 
                 \end{cases}$$
            which implies, for $0 \leq j \leq n-1 $, that 
                \begin{align*}
                    \int_{A_1}\xi(|x|)x_n[\Hess (\mu  w_s + \lambda v_t)(x)]_j\d x&=\mu^j\binom{n-1}{j}\int_{s}^{t} r^{n-j} \xi(r) \d r \int_{\sn \cap {y_n >0}}y_n \d y\  e_n \\
                    &=\mu^j\binom{n-1}{j} \frac{\omega_{n-1}}{n-1} \int_{s}^{t} r^{n-j} \xi(r) \d r. 
                \end{align*}    
    
            \item The function $\mu  w_s + \lambda v_t$ is of class $C^2$ on $A_2=\{x \in \R^{n} : |x| > t, x_n > 0\}$, and its Hessian matrix at a point $x$ in this set has $(n-1)$-eigenvalues equal to $(\mu+\lambda)/|x|$ and the last eigenvalue is equal to zero. Hence,
            $$[\Hess  (\mu  w_s + \lambda v_t)(x)]_j=\begin{cases}
                     \binom{n-1}{j}\frac{(\mu+\lambda)^j}{|x|^j}  & \text{if}\ 0 \leq j \leq n-1 \\
                    0 & \text{if}\ j=n 
                 \end{cases}$$
            which implies, for $0 \leq j \leq n-1 $, that
            \begin{align*}
                \int_{A_2}\xi(|x|)&x_n[\Hess (\mu  w_s + \lambda v_t)(x)]_j\d x  \\&=(\mu+\lambda)^j\binom{n-1}{j}\int_{t}^{\infty} r^{n-j} \xi(r) \d r \int_{\sn \cap {y_n > 0}}y_n \d y\\
                &=(\mu+\lambda)^j\binom{n-1}{j}\frac{\omega_{n-1}}{n-1} \int_{t}^{\infty} r^{n-j} \xi(r) \d r.
            \end{align*}
             \item The function $\mu  w_s + \lambda v_t $ is of class $C^2$ on $A_3=\{x \in \R^{n}: |x|_{n-1} < s, |x| > t, x_n < 0\}$ and its Hessian matrix at a point $x$ in this set has $(n-1)$-eigenvalues equal to $\lambda/|x|$ and the last eigenvalue is equal to zero. Hence,
              $$[\Hess  (\mu  w_s + \lambda v_t)(x)]_j=\begin{cases}
                     \binom{n-1}{j}\frac{\lambda^j}{|x|^j}  & \text{if}\ 0 \leq j \leq n-1 \\
                    0 & \text{if}\ j=n 
                 \end{cases}$$
            which implies, for $0 \leq j \leq n-1 $, that
            \begin{align*}
                \int_{A_3}&\xi(|x|)x_n[\Hess (\mu  w_s + \lambda v_t)(x)]_j\d x
                \\
                &= \lambda^j\binom{n-1}{j}  \int_{0}^{+\infty}r^{n-j}\xi(r) \int_{\sn \cap e_n^{\perp}}\int_{-1}^{1} \tau (1-\tau^2)^{\frac{n-3}{2}} \mathds{1}_{A_3} \d\tau \d y \d r\\
                &=\lambda^j \binom{n-1}{j}\omega_{n-1} \int_{t}^{+\infty}r^{n-j}\xi(r) \int_{-1}^{-\sqrt{1-\frac{s^2}{r^2}}} \tau (1-\tau^2)^{\frac{n-3}{2}}\d\tau \d r\\
                &=-\lambda^j \binom{n-1}{j} \frac{\omega_{n-1}}{n-1} s^{n-1} \int_{t}^{+\infty}r^{1-j}\xi(r)\d r.
            \end{align*}
             \item The function $\mu  w_s + \lambda v_t $ is of class $C^2$ on $A_4=\{x \in \R^n: |x|_{n-1} > s, |x| < t, x_n < 0\}$ and the eigenvalues of its Hessian at a point $x$ in this set are given by 
                $$[\Hess  (\mu  w_s + \lambda v_t)(x)]_j=\begin{cases}
                     \binom{n-2}{j} \frac{ \mu^j}{|x|^j_{n-1}} & \text{if}\ 0 \leq j \leq n-2 \\
                    0 & \text{if}\ j\in \{n-1,n\} 
                 \end{cases}$$
             which implies, for $0 \leq j \leq n-2 $, that
             \begin{align*}
                \int_{A_4}\xi(|x|)x_n&[\Hess (\mu  w_s + \lambda v_t)(x)]_j\d x \\
                & =\mu^j\binom{n-2}{j}  \int_{0}^{+\infty}r^{n-j}\xi(r) \int_{\sn \cap e_n^{\perp}}\int_{-1}^{1} \tau (1-\tau^2)^{\frac{n-j-3}{2}} \mathds{1}_{A_4} \d\tau \d y \d r\\
                &= \mu^j\binom{n-2}{j}\omega_{n-1}  \int_{s}^{t}r^{n-j}\xi(r)\int_{-\sqrt{1-\frac{s^2}{r^2}}}^{0} \tau (1-\tau^2)^{\frac{n-j-3}{2}}\d\tau \d r \\
                &=- \mu^j\binom{n-2}{j}\frac{\omega_{n-1}}{n-j-1}  \left[  \int_{s}^{t}r^{n-j}\xi(r) \d r - s^{n-j-1} \int_{s}^{t}r\xi(r)\d r \right]\\
                &=- \mu^j\binom{n-1}{j}\frac{\omega_{n-1}}{n-1}  \left[  \int_{s}^{t}r^{n-j}\xi(r) \d r - s^{n-j-1} \int_{s}^{t}r\xi(r)\d r \right].
            \end{align*}
             \item The function $\mu  w_s + \lambda v_t $ is of class $C^2$ on $A_5=\{x \in \R^n: |x|_{n-1} > s, |x| > t, x_n < 0\}$ and the eigenvalues of its Hessian at a point $x$ in this set are given by  $$[\Hess (\mu  w_s + \lambda v_t)(x)]_j= \begin{cases}
                 \binom{n-2}{j}\left(\frac{\lambda}{|x|} + \frac{\mu}{|x|_{n-1}}\right)^j \\
                 + & \text{if}\  0 \leq j \leq n-2 \\
                 \binom{n-2}{j-1}\left(\frac{\lambda}{|x|} + \frac{\mu}{|x|_{n-1}}\right)^{j-1}\frac{\lambda}{|x|} \\
                 \\
                \left(\frac{\lambda}{|x|} + \frac{\mu}{|x|_{n-1}}\right)^{n-2}\frac{\lambda}{|x|} & \text{if}\  j=n-1 \\
                \\
                0 & \text{if}\  j=n
             \end{cases}$$
            this implies, for $0 \leq j \leq n-2 $, that
            \begin{align*}
                &\int_{A_5}\xi(|x|)x_n[\Hess (\mu  w_s + \lambda v_t)(x)]_j\d x= \\
                &\omega_{n-1}\left[\binom{n-2}{j}\sum_{k=0}^{j}\binom{j}{k}\mu^k \lambda^{j-k}\int_{t}^{+\infty}r^{n-j}\xi(r) \int_{-\sqrt{1-\frac{s^2}{r^2}}}^{0} \tau (1-\tau^2)^{\frac{n-k-3}{2}} \d\tau \d r \ +\right. \\
                &\left. \binom{n-2}{j-1}\sum_{k=0}^{j-1}\binom{j-1}{k}\mu^k \lambda^{j-k}\int_{t}^{+\infty}r^{n-j}\xi(r) \int_{-\sqrt{1-\frac{s^2}{r^2}}}^{0} \tau (1-\tau^2)^{\frac{n-k-3}{2}} \d\tau \d r   \right]\\
                &=-\omega_{n-1} \times \\
                & \left[\binom{n-2}{j}\sum_{k=0}^{j}\binom{j}{k}\frac{\mu^k \lambda^{j-k}}{n-k-1}\left(\int_{t}^{+\infty}r^{n-j}\xi(r)\d r - s^{n-k-1} \int_{t}^{+\infty}r^{k-j+1}\xi(r)\d r \right)+ \right. \\
                &\left.\binom{n-2}{j-1}\sum_{k=0}^{j-1}\binom{j-1}{k}\frac{\mu^k \lambda^{j-k}}{n-k-1}\left(\int_{t}^{+\infty}r^{n-j}\xi(r)\d r - s^{n-k-1} \int_{t}^{+\infty}r^{k-j+1}\xi(r)\d r \right) \right] \\
                &= - \frac{\omega_{n-1}}{n-1}\binom{n-1}{j} \times \\
                &\left[\sum_{k=0}^{j}\binom{j}{k}\mu^k \lambda^{j-k}\left(\int_{t}^{+\infty}r^{n-j}\xi(r)\d r - s^{n-k-1} \int_{t}^{+\infty}r^{k-j+1}\xi(r)\d r \right) \right],
            \end{align*}
            and
              \begin{align*}
                &\int_{A_5}\xi(|x|)x_n[\Hess (\mu  w_s + \lambda v_t)(x)]_{n-1}\d x\\
                &= \omega_{n-1}\sum_{k=0}^{n-2}\binom{n-2}{k}\mu^k \lambda^{n-k-1}\int_{t}^{+\infty}r\xi(r) \int_{-\sqrt{1-\frac{s^2}{r^2}}}^{0} \tau (1-\tau^2)^{\frac{n-k-3}{2}} \d\tau \d r\\
                &= -\omega_{n-1}\sum_{k=0}^{n-2}\binom{n-2}{k}\frac{\mu^k \lambda^{n-k-1}}{n-k-1}\left(\int_{t}^{+\infty}r\xi(r)\d r - s^{n-k-1} \int_{t}^{+\infty}r^{k-n+2}\xi(r)\d r \right)\\
                &=-\frac{\omega_{n-1}}{n-1}\sum_{k=0}^{n-2}\binom{n-1}{k}\mu^k \lambda^{n-k-1}\left(\int_{t}^{+\infty}r\xi(r)\d r - s^{n-k-1} \int_{t}^{+\infty}r^{k-n+2}\xi(r)\d r\right),
            \end{align*}
            for $j=n-1$.
             \item For the first set of singular points $B_1= \{x \in \R^n  : |x| = s, \ x_n > 0\}$, set $U_1 = \{x \in\R^n: |x|< t, x_n > 0 \}$. Then $U_1$ is an open set and $B_1 \subset U_1$. By Lemma \ref{sum_of_two_functions}, for every $x \in U_1$,
             \[\mu w_s + \lambda v_t = \mu v_s.\]
             Let $ E \subset B_1$ be a Borel set. Since $B_1 \subset U_1$, we have $E \subset U_1$. Hence, by the local determination property of Hessian measures (\ref{locally_determined}) on the open set $U_1$, it follows that
             \[\Phi^n_j(\mu  w_s + \lambda v_t; E)= \Phi^n_j(\mu v_s; E).\]
             Since this holds for every Borel set $E \subset B_1$, we obtain 
             \[\Phi^n_j(\mu  w_s + \lambda v_t; \cdot)|_{B_1} = \Phi^n_j(\mu v_s;\cdot)|_{B_1}= \mu^j \Phi^n_j(v_s;\cdot)|_{B_1}.\]
             Since $v_s$ is radial, this implies that  $\Phi^n_j( v_s,\cdot)|_{\{|x|=s\}}$ is  rotation invariant. Hence, it is proportional to the restriction of $\hm^{n-1}$ to $|x|=s$. Thus, restricting to $B_1$ there exists a constant $C_{n,j,s}$ such that,
    \begin{align*}
        \int_{B_1} \gamma(x) \d\Phi^n_j(v_s;x)= C_{n,j,s} \int_{B_1}\gamma(x)\d\hm ^{n-1}(x)
    \end{align*}
    for every $\gamma \in C_b(\R^n)$, which implies that
    \begin{align*}
        \Phi^n_j(v_s;B_1)&= C_{n,j,s}\int_{B_1}\d\hm ^{n-1}(x) \\
        &= \frac{\omega_{n}}{2} C_{n,j,s}\  s^{n-1} 
    \end{align*}
    and by a similar proof to Lemma \ref{retrieving_densities_for_volumes}, we have that 
    \begin{align*}
        \Phi^n_j(v_s;B_1)=\frac{ \kappa_{n}}{2}\binom{n}{j}s^{n-j}
    \end{align*}
    for every $1 \leq j \leq n$. Hence,
    \begin{align*}
        C_{n,j,s}=\frac{1}{n}\binom{n}{j}s^{1-j}
    \end{align*}
    for every $1 \leq j \leq n$ and,
    \begin{align*}
        \int_{B_1}\xi(|x|)x_n\d\Phi^n_j(\mu  w_s + \lambda v_t;x)&= \mu^j\int_{B_1}\xi(|x|)x_n\d\Phi^n_j(v_s;x) \\
        &= \frac{1}{n}\mu^j\binom{n}{j}s^{1-j}\int_{B_1} \xi(|x|)x_n \d\hm ^{n-1}(x) \\
        &=\frac{1}{n}\mu^j \binom{n}{j}s^{n-j+1}\xi(s)\int_{\sn \cap \{x_n > 0\}}x_n \d\hm ^{n-1}(x)\\
        &=\frac{1}{n} \mu^j\binom{n}{j}\frac{\omega_{n-1}}{n-1}s^{n-j+1}\xi(s).
    \end{align*}
    \item For the second set of singular points $B_2= \{x \in \R^n: |x| = t, x_n >0\}$. Since $\mu  w_s + \lambda v_t$ is radial in a neighborhood of $B_2$, by the local determination property of the Hessian measures, the restriction $\Phi^n_j(\mu  w_s + \lambda v_t;\cdot)|_{B_2}$ is proportional to $\hm^{n-1}|_{B_2}$. Thus, there exists a constant $C_{n,j,t}$ such that,
    \begin{align*}
        \int_{B_2} \gamma(x) \d\Phi^n_j(\mu  w_s + \lambda v_t;x)= C_{n,j,t} \int_{B_2}\gamma(x)\d\hm ^{n-1}(x)
    \end{align*}
    for every $\gamma \in C_b(\R^n)$. It follows that,
    \begin{align*}
        \Phi^n_j(\mu w_s + \lambda v_t;B_2)&= C_{n,j,t}\int_{B_2}\d\hm ^{n-1}(x) \\
        &=\frac{\omega_{n}}{2} C_{n,j,t}\  t^{n-1} 
    \end{align*}
    and on the other hand, by Lemma \ref{singular_point_addition},
    \begin{align*}
        \partial (\mu w_s + \lambda v_t)(x) =  \Bigl\{ \alpha \frac{x}{|x|} : \alpha \in [\mu,\mu+\lambda]\Bigl\}
    \end{align*}
    for every $x \in B_2$, which implies that 
    \begin{align*}
        P_p(\mu w_s + \lambda v_t, B_2) =  \Bigl\{ (t  + \alpha)x  \ | \ \alpha \in [\mu p,(\mu+\lambda)p], x\in \sn, \ x_n > 0\Bigl\}.
    \end{align*}
   Thus, for every $1 \leq j \leq n$,
    \begin{align*}
        \Phi^n_j(\mu  w_s + \lambda v_t;B_2)&= \binom{n}{j}\frac{\kappa_{n}}{2} t^{n-j} \sum_{k=0}^{j-1}\binom{j}{k}\mu^{k}\lambda^{j-k}.
    \end{align*}
   It follows that, 
    \begin{align*}
        C_{n,j,t}=\frac{1}{n}\binom{n}{j}t^{1-j} \sum_{k=0}^{j-1}\binom{j}{k}\mu^{k}\lambda^{j-k} 
    \end{align*}
   for every $1 \leq j\leq n$. Thus,
    \begin{align*}
        \int_{B_2}&\xi(|x|)x_n\d\Phi^n_j(\mu  w_s + \lambda v_t;x)
        \\
        &=\frac{1}{n}\binom{n}{j}t^{1-j}\sum_{k=0}^{j-1}\binom{j}{k}\mu^{k}\lambda^{j-k}  \int_{B_2} \xi(|x|)x_n \d\hm ^{n-1}(x) \\
        &=\frac{1}{n}\binom{n}{j}t^{n-j+1}\xi(t)\sum_{k=0}^{j-1}\binom{j}{k}\mu^{k}\lambda^{j-k} \int_{\sn \cap \{x_n >0\}}x_n \d\hm ^{n-1}(x)\\
        &=\frac{1}{n}\frac{\omega_{n-1}}{n-1}\binom{n}{j}t^{n-j+1}\xi(t)\sum_{k=0}^{j-1}\binom{j}{k}\mu^{k}\lambda^{j-k}.
    \end{align*}
    \item For the third set of singular points $B_3 = \{x \in \R^n: |x|_{n-1}<s, |x| = t, x_n < 0\}$, set $U_3 = \{x\in\R^n : |x|_{n-1}<s,\ x_n<0\}$. Then $U_3$ is open and $B_3\subset U_3$. Moreover, on $U_3$, we have
    \[
    \mu w_s+\lambda v_t=\lambda v_t.
    \]
    Let $E\subset B_3$ be a Borel set. Since $B_3\subset U_3$, we have $E\subset U_3$. Hence, by the local determination property of Hessian measures (\ref{locally_determined}) applied on the open set $U_3$, we have
    \begin{align*}
        \Phi^n_j(\mu  w_s + \lambda v_t;E) =\Phi^n_j(\lambda v_t;E).
    \end{align*}
    Since this holds for every Borel set $E\subset B_3$, we obtain
    \[ \Phi^n_j(\mu  w_s + \lambda v_t;\cdot)|_{B_3} =\Phi^n_j(\lambda v_t;\cdot)|_{B_3}= \lambda^j\Phi^n_j(v_t;\cdot)|_{B_3}.\]
    Since $ v_t$ is radially symmetric, this implies that $\Phi^n_j( v_t;\cdot)|_{\{|x|=t\}}$ is rotation invariant. Hence, restricting to $B_3$, there exists a constant $K_{n,j,t}$ such that, 
    \begin{align*}
        \int_{B_3} \gamma(x) \d\Phi^n_j( v_t;x)= K_{n,j,t} \int_{B_3 }\gamma(x)\d\hm ^{n-1}(x)
    \end{align*}
    for every $\gamma \in C_b(\R^n)$. We obtain that,
    \begin{align*}
        \Phi^n_j(v_t;B_3)&= K_{n,j,t}\int_{B_3}\d\hm ^{n-1}(x) \\
        &=\omega_{n}K_{n,j,t}\ t^{n-1} \int_{-1}^{-\sqrt{1-\frac{s^2}{t^2}}} (1-\tau^2)^{\frac{n-3}{2}}\d\tau.
    \end{align*}
    On the other hand, using a similar proof of Lemma \ref{retrieving_densities_for_volumes}, for every $ 1 \leq j \leq n$, one can show that
    \begin{align*}
        \Phi^n_j(v_t;B_3)= \kappa_n \binom{n}{j}t^{n-j}\int_{-1}^{-\sqrt{1-\frac{s^2}{t^2}}} (1-\tau^2)^{\frac{n-3}{2}}\d\tau
    \end{align*}
    It follows that, 
    \begin{align*}
       K_{n,j,t} = \frac{1}{n}\binom{n}{j}  t^{1-j}.
    \end{align*}
    Hence,
    \begin{align*}
        \int_{B_3}&\xi(|x|)x_n\d\Phi^n_j(\mu  w_s + \lambda v_t;x)
        \\&= \frac{1}{n}\binom{n}{j} \lambda^j t^{1-j} \int_{B_3} \xi(|x|)x_n \d\hm ^{n-1}(x)\\
        &=\frac{1}{n}\binom{n}{j} \lambda^j  t^{n-j+1}\xi(t)\int_{\sn \cap |x|_{n-1} < \frac{s}{t} \cap x_n < 0}x_n \d\hm ^{n-1}(x)\\
       & = \frac{1}{n}\binom{n}{j} \lambda^j  t^{n-j + 1}\xi(t) \int_{\sn \cap e_n^{\perp}} \int_{-1}^{-\sqrt{1-\frac{s^2}{t^2}}} \tau(1-\tau^2)^{\frac{n-3}{2}} \d\tau \d\omega\\
        & = -\frac{1}{n}\frac{\omega_{n-1}}{n-1} \binom{n}{j} \lambda^j t^{2-j} s^{n-1} \xi(t).
    \end{align*}
    \item For the fourth set of singular points $B_4= \{x \in \R^n : |x|_{n-1} = s,|x| < t, x_n < 0\}$, let 
    \begin{align*}
        A = \{x \in \R^n: |x| < t, x_n < 0\}
    \end{align*}
    by Lemma \ref{sum_of_two_functions}, for every $x \in A$, we have
        \begin{align*}
            (\mu w_s + \lambda v_t)(x_1,\dots,x_n) = \mu \max\{0,|x|_{n-1}-s\}.
        \end{align*} 
   Thus, on the open set $A$, the function $(\mu w_s + \lambda v_t)(x_1,\dots,x_n)$ is locally determined by the function
   \begin{align*}
       x \mapsto \mu \max\{0,|x|_{n-1}-s\}.
   \end{align*}
   Consequently, by the local determination property (\ref{locally_determined}), the corresponding Hessian measures coincide on every Borel subset of $A$, and in particular for $B_4$.
   It follows from Lemma \ref{addition_complementary}, for every $1 \leq j \leq n -1$, that
    \begin{align*}
        \int_{B_4}\xi(|x|)x_n&\d\Phi^n_j(\mu  w_s + \lambda v_t;x)\\&= \mu^j\int_{B_4} \xi(|x|)x_n\d\Phi^{n-1}_j( \max\{0,|x|_{n-1}-s\};(x_1,\dots,x_{n-1}))\d x_n.
    \end{align*}
    By Lemma \ref{retrieving_densities_for_volumes}, 
    \begin{align*}
        \Phi^{n-1}_j(\max\{0,|x|_{n-1}-s\};s\s^{n-2})=\frac{\omega_{n-1}}{n-1}\binom{n-1}{j}s^{n-j-1}
    \end{align*}
    and hence for every $1 \leq j\leq n-1$,
    \begin{align*}
        \int_{B_4}\xi(|x|)&x_n\d\Phi^n_j(\mu  w_s + \lambda v_t;x)
         \\&=\mu^j\int_{B_4} \xi(|x|)x_n\d\Phi^{n-1}_j( \max\{0,|x|_{n-1}-s\};\left(x_,\cdots,x_{n-1}\right))\d x_n\\ 
         &=\mu^j\binom{n-1}{j}\frac{\omega_{n-1}}{n-1}s^{n-j-1}\int_{\sqrt{s^2 + x_n^2}< t \,\cap\, x_n < 0}\xi\left(\sqrt{s^2+x_n^2}\right)x_n\d x_n\\
         &=- \mu^j\binom{n-1}{j}\frac{\omega_{n-1}}{n-1}s^{n-j-1}\int_{s}^{t}r\xi(r) \d r.
    \end{align*}
   
    \item For the fifth set of singular points given by $B_5= \{x \in \R^n : |x|_{n-1} = s, |x| >  t, x_n < 0\}$, we remark that $x_n = - \sqrt{|x|^2 - s^2}$. By (\ref{singular_point_addition}), we have
    \begin{align*}
        \partial( \mu w_s + \lambda v_t)(x) = \partial (\mu w_s)  + \partial(\lambda v_t)=  \Bigl\{  \left(\alpha \frac{x_1}{|x|_{n-1}},\dots,\alpha \frac{x_{n-1}}{|x|_{n-1}},0\right): \alpha \in [0,\mu]\Bigl\} + \lambda \frac{x}{|x|}
    \end{align*}
   for every $ x \in B_5$. Therefore the parallel set of $B_5$ is as follows,
     \begin{align*}
         &P_p(\mu  w_s + \lambda v_t,B_5)=\\
         & \Bigl\{ x + p \alpha \frac{(x - x_n e_n)}{|x|_{n-1}} + p \lambda \frac{x}{|x|} :  |x| > t, |x|_{n-1} = s,   x_n = - \sqrt{|x|^2 - s^2}, \alpha \in [0,\mu] \Bigl\}=\\
         &\Bigl\{ (r + p\lambda)(\tau e_n + \sqrt{1-\tau^2}v) + p \alpha v : v \in \sn\cap e_n^{\perp}, r > t,  \tau = -\frac{1}{r}\sqrt{r^2-s^2}, \alpha \in [0,\mu]\Bigl\}
     \end{align*}
     and thus 
     \begin{align*}
          &P_p(\mu  w_s + \lambda v_t,B_5)=\\
         &\Bigl\{ -\left(1 + p\frac{\lambda}{r}\right)\sqrt{r^2-s^2} e_n + \left(s + p \left(\frac{\lambda s}{r} + \alpha\right)\right)v : v \in \sn\cap e_n^{\perp},  r > t,  \alpha \in [0,\mu]\Bigl\}
     \end{align*}
     for all $p \geq 0$. Hence, the $n$-dimensional Hausdorff measure of the parallel set $P_p(\mu  w_s + \lambda v_t,B_5)$ is given by 
        \begin{align*}
            &\omega_{n-1}\int_{t}^{+\infty}\int_{0}^{\mu} \left[p\lambda\frac{s^2}{r^2\sqrt{r^2-s^2}} + \frac{r}{\sqrt{r^2-s^2}}\right] p \left(s + p \left(\frac{\lambda s}{r} + \alpha\right)\right)^{n-2} \d\alpha\d r=\frac{\omega_{n-1}}{n-1} \times\\
            &\int_{t}^{+\infty}\left[p\lambda\frac{s^2}{r^2\sqrt{r^2-s^2}} + \frac{r}{\sqrt{r^2-s^2}}\right]\left[ \left(s + p \left(\frac{\lambda s}{r} + \mu\right)\right)^{n-1} - \left(s + p\frac{\lambda s}{r} \right)^{n-1} \right]\d r=\\
            &\frac{\omega_{n-1}}{n-1}\sum_{j=1}^{n-1}p^{j} \binom{n-1}{j}\sum_{k=1}^{j}\binom{j}{k}\lambda^{j-k}\mu^{k}s^{n-k-1} \times \\
            &\int_{t}^{+\infty}\left[p\lambda\frac{s^2}{r^2\sqrt{r^2-s^2}} + \frac{r}{\sqrt{r^2-s^2}}\right]\frac{1}{r^{j-k}} \d r=\\
            &\frac{\omega_{n-1}}{n-1}\sum_{j=1}^{n-1}p^{j+1}\binom{n-1}{j}\sum_{k=1}^{j}\binom{j}{k} \lambda^{j-k+1}\mu^k  s^{n-k+1} \int_{t}^{+\infty}\frac{1}{r^{j-k+2}\sqrt{r^2-s^2} }\d r +  \\
            & \frac{\omega_{n-1}}{n-1}\sum_{j=1}^{n-1} p^{j} \binom{n-1}{j} \sum_{k=1}^{j}\binom{j}{k}\lambda^{j-k}\mu^k s^{n-k-1}\int_{t}^{+\infty}\frac{1}{r^{j-k-1}\sqrt{r^2-s^2} }\d r. 
        \end{align*}
    Thus, for $j =1$,
    \begin{align*}
        \Phi^n_1(\mu  w_s + \lambda v_t; B_5)&= \omega_{n-1}s^{n-2}\mu\int_{t}^{+\infty}\frac{r}{\sqrt{r^2-s^2}}\d r \\
        &= \omega_{n-1}s^{n-2}\mu\int_{-\infty}^{-\sqrt{t^2-s^2}}\d x_n
     \end{align*}
     and for $2 \leq j \leq n-1$,
     \begin{align*}
         \Phi^n_j&(\mu  w_s + \lambda v_t; B_5)\\
         &= \frac{\omega_{n-1}}{n-1} \left[ \binom{n}{j}\frac{1}{n}\sum_{k=1}^{j-1}\binom{j}{k} (j-k) s^{n-k+1}\lambda^{j-k}\mu^k  \int_{t}^{+\infty}\frac{1}{r^{j-k+1}\sqrt{r^2-s^2} }\d r \right.\\
         &\left. +\binom{n-1}{j} \sum_{k=1}^{j}\binom{j}{k}s^{n-k-1}\lambda^{j-k}\mu^k  \int_{t}^{+\infty}\frac{1}{r^{j-k-1}\sqrt{r^2-s^2} }\d r\right] \\
         &= \frac{\omega_{n-1}}{n-1} \left[ \binom{n}{j}\frac{1}{n}\sum_{k=1}^{j-1}\binom{j}{k} (j-k) s^{n-k+1}\lambda^{j-k}\mu^k  \int_{-\infty}^{-\sqrt{t^2-s^2}}(x_n^2+s^2)^{\frac{k-j-2}{2}}\d x_n \right.\\
         &\left. +\binom{n-1}{j} \sum_{k=1}^{j}\binom{j}{k}s^{n-k-1}\lambda^{j-k}\mu^k  \int_{-\infty}^{-\sqrt{t^2-s^2}}(x_n^2+s^2)^{\frac{k-j}{2}}\d x_n\right].
     \end{align*}
     Lastly, for $j =n$, 
     \begin{align*}
          \Phi^n_{n}(\mu  w_s + \lambda v_t; B_5)&=\frac{\omega_{n-1}}{n-1}\sum_{k=1}^{n-1}\binom{n-1}{k}  s^{n-k+1}\lambda^{n-k}\mu^k  \int_{t}^{+\infty}\frac{1}{r^{n-k+1}\sqrt{r^2-s^2} }\d r\\
          &=\frac{\omega_{n-1}}{n-1}\sum_{k=1}^{n-1}\binom{n-1}{k}  s^{n-k+1}\lambda^{n-k}\mu^k  \int_{-\infty}^{-\sqrt{t^2-s^2}}(x_n^2+s^2)^{\frac{k-j-2}{2}}\d x_n.
     \end{align*}
    Since $\xi$ has bounded support, there exists $R>t$ such that $\xi(r)=0$ for $r\geq R$. Hence the integral over $B_5$ only depends on the restriction of
     \[
     \nu_j:=\Phi_j^n(\mu w_s+\lambda v_t;\cdot)|_{B_5}
     \]
   to the bounded set $B_5\cap \{|x|\leq R\}$, on which this measure is finite. Thus we may apply Theorem \ref{disintegration_theorem} to this finite restriction of $\nu_j$, and we keep the notation $\nu_j$ for simplicity. Now let
     \begin{align*}
         \pi: \R^n &\rightarrow \R\\
         x &\rightarrow \pi(x)= x_n,
     \end{align*}
     and set $\Sigma = \left(-\infty,-\sqrt{t^2-s^2}\right) $. Since,
     \[B_5 =  \{x \in \R^n : |x|_{n-1} = s, |x| >  t, x_n < 0\}, \]
     then for every $x_n \in \Sigma$, 
     \[B_5 \cap \pi^{-1}(x_n) = s\s^{n-2}\times \{x_n\}.\]
   By Theorem \ref{disintegration_theorem} applied to the finite measure $\nu_j$ and to the projection $\pi$, there exists a family of measures $(\nu_{j,x_n})_{x_n \in \Sigma}$ defined for $(\pi_{\#}\nu_j)$-almost every $x_n$ such that $\nu_{j,x_n}$ is concentrated on $B_5 \cap \pi^{-1}(x_n)$, and
     \begin{align*}
        \int_{B_5} \varphi(x)\d\Phi^n_j(\mu  w_s + \lambda v_t;x) = \int_{\Sigma}\int_{s\s^{n-2}} \d \eta(x_n) \varphi(x{'},x_n) \d \nu_{j,x_n}(x{'})\d(\pi_{\#}\nu_j)(x_n)
     \end{align*}
     for every bounded Borel function $\varphi$ supported in $\{|x|\le R\}$. Since the function $\mu  w_s + \lambda v_t$ is $\SO(n-1)$-invariant on $B_5 \cap \pi^{-1}(x_n)$ for every $x_n \in \Sigma$, the measure $\nu_{j,x_n}$ is $\SO(n-1)$-invariant on each fiber. Therefore, it must be a scalar multiple of the Hausdorff measure $\hm^{n-2}$ on $s\s^{n-2}$. This implies that there exist a function $K_{n,j}$ depending on $x_n$ such that 
     \begin{align}\label{disintegration_hessian_B_5}
         &\int_{B_5} \varphi(x)\d\Phi^n_j(\mu  w_s + \lambda v_t;x)=\int_{\Sigma}K_{n,j}(x_n)\int_{s\mathbb{S}^{n-2}}\varphi(x^{'},x_n) \d\hm ^{n-2}(x^{'})\pi_{\#}\nu_j(x_n),
    \end{align}
     and since for every generating Borel set of $\R$, which is given by $(-\infty,c)$ for some $c \in \R$, we have
     \begin{align*}
         (\pi_{\#}\nu_j)(Y) = \Phi^n_j(\mu  w_s + \lambda v_t;B_5 \cap \pi^{-1}(Y)).    \end{align*}
     Then, by computing the parallel set of the function $\mu  w_s + \lambda v_t$ at $B_5 \cap \pi^{-1}(Y)$, we have that for $j=1$,
     \begin{align*}
         \Phi^n_1(\mu  w_s + \lambda v_t;B_5 \cap \pi^{-1}(Y)) = &\omega_{n-1}s^{n-2}\mu\int_{t}^{+\infty}\mathds{1}_Y\left(-\sqrt{r^2-s^2}\right)\frac{r}{\sqrt{r^2-s^2}}\d r \\
         &=\omega_{n-1}s^{n-2}\mu\int_{-\infty}^{-\sqrt{t^2-s^2}}\mathds{1}_Y(x_n)\d x_n
     \end{align*}
     while for $2 \leq j \leq n-1$,
     \begin{align*}
          &\Phi^n_j(\mu  w_s + \lambda v_t;B_5 \cap \pi^{-1}(Y))=\frac{\omega_{n-1}}{n-1} \times \\
         & \left[ \binom{n}{j}\frac{1}{n}\sum_{k=1}^{j-1}\binom{j}{k} (j-k) s^{n-k+1}\lambda^{j-k}\mu^k  \int_{-\infty}^{-\sqrt{t^2-s^2}}\mathds{1}_Y(x_n)(x_n^2+s^2)^{\frac{k-j-2}{2}}\d x_n +\right.\\
         &\left. \binom{n-1}{j} \sum_{k=1}^{j}\binom{j}{k}s^{n-k-1}\lambda^{j-k}\mu^k   \int_{-\infty}^{-\sqrt{t^2-s^2}}\mathds{1}_Y(x_n)(x_n^2+s^2)^{\frac{k-j}{2}}\d x_n \right]
     \end{align*}
     and 
     \begin{align*}
          &\Phi^n_n(\mu  w_s + \lambda v_t;B_5 \cap \pi^{-1}(Y))=\\
          &\frac{\omega_{n-1}}{n-1}\sum_{k=1}^{n-1}\binom{n-1}{k}  s^{n-k+1}\lambda^{n-k}\mu^k  \int_{-\infty}^{-\sqrt{t^2-s^2}}\mathds{1}_Y(x_n)(x_n^2+s^2)^{\frac{k-n-2}{2}}\d x_n 
     \end{align*}
     when $j =n$. This implies that the measure $\pi_{\#}\nu_j$ is absolutely continuous with respect to the Lebesgue measure. After substituting for $\varphi = 1$ in (\ref{disintegration_hessian_B_5}) and comparing with the total measure of $\nu_j$, we get that $K_{n,j}(x_n) = \frac{s^{2-n}}{\omega_{n-1}}$ for every $x_n \in \R$. Thus, for $j = 1$, 
    \begin{align*}
         \int_{B_5} \xi(|x|)x_n\d\Phi^n_1(\mu  w_s + \lambda v_t;x) 
         &= \omega_{n-1} \mu  s^{n-2} \int_{\Sigma}\int_{s\mathbb{S}^{n-2}}\xi(|x|) x_n \d\hm ^{n-2}(x^{'})\d x_n \\
         & = \omega_{n-1} \mu  s^{n-2}\int_{-\infty}^{-\sqrt{t^2-s^2}} \xi\left(\sqrt{s^2+x_n^2}\right)x_n \d x_n\\
         &= -\omega_{n-1}\mu  s^{n-2}\int_{t}^{+\infty} r\xi(r) \d r 
    \end{align*}
    and for every $1 \leq j \leq n-1$,
    \begin{align*}
         \int_{B_5} &\xi(|x|)x_n\d\Phi^n_j(\mu  w_s + \lambda v_t;x) 
         \\=&-\frac{\omega_{n-1}}{n-1} \left[ \binom{n}{j}\frac{1}{n}\sum_{k=1}^{j-1}\binom{j}{k} (j-k) \lambda^{j-k}\mu^k s^{n-k+1}\int_{t}^{+\infty}r^{k-j-1}\xi(r)\d r  \right.\\
         &\left.+\binom{n-1}{j} \sum_{k=1}^{j}\binom{j}{k}\lambda^{j-k}\mu^k s^{n-k-1} \int_{t}^{+\infty}r^{k-j+1}\xi(r)\d r\right]
    \end{align*}
    whereas for $j =n$, we get, 
    \begin{align*}
        \int_{B_5} \xi(|x|)&x_n\d\Phi^n_n(\mu  w_s + \lambda v_t;x) 
         \\&=-\frac{\omega_{n-1}}{n-1} \sum_{k=1}^{n-1}\binom{n-1}{k}\lambda^{n-k}\mu^k  s^{n-k+1} \int_{t}^{+\infty} r^{k-n-1}\xi(r) \d r.
    \end{align*}
    \item For the sixth set of singular points given by $B_6=\{x \in \R^n : |x|=t, |x|_{n-1} > s, x_n < 0\}$, we have by (\ref{singular_point_addition}),
    \begin{align*}
        \partial (\mu \ w_s + \lambda \ v_t)(x)&=\partial (\mu \ w_s)(x)  + \partial( \lambda \ v_t)(x)\\
        &= \Bigl\{  \left(\mu \frac{x_1}{|x|_{n-1}},\dots,\mu \frac{x_{n-1}}{|x|_{n-1}},0\right) \Bigl\} + \Bigl\{  \alpha\frac{x}{|x|} : \alpha \in [0,\lambda] \Bigl\}
    \end{align*}
     for every $x \in B_6$  , and thus 
     \begin{align*}
         &P_p(\mu  w_s + \lambda v_t, B_6)\\&=
         \Bigl\{ x + p \mu \frac{(x - x_n e_n)}{|x|_{n-1}} + p \alpha \frac{x}{|x|}: |x|= t, |x|_{n-1} > s,  x_n < 0, \alpha \in [0,\lambda] \Bigl\}\\
         & =\Bigl\{ (t + p\alpha)(\tau e_n + \sqrt{1-\tau^2}v) + p\mu v: v \in \sn\cap e_n^{\perp}, \tau \in [-\sqrt{1-\frac{s^2}{t^2}},0], \alpha \in [0,\lambda]\Bigl\}
     \end{align*}
     for all $p \geq 0$. 
     
     Hence, the $n$-dimensional Hausdorff measure of the parallel set associated with $B_6$ is given by 
     \begin{align*}
         &\omega_{n-1}\int_{0}^{\lambda} \int_{-\sqrt{1-\frac{s^2}{t^2}}}^{0}p(t+p\alpha)^{n-1} \left(1 + \frac{p\mu}{(t + p\alpha)\sqrt{1-\tau^2}}\right)^{n-2} (1 - \tau^2)^{\frac{n-3}{2}} \d\tau \d\alpha=\\
         &\omega_{n-1}\int_{0}^{\lambda} \int_{-\sqrt{1-\frac{s^2}{t^2}}}^{0}p(t+p\alpha) \left(t+p\left(\alpha + \frac{\mu}{\sqrt{1-\tau^2}}\right)\right)^{n-2} (1 - \tau^2)^{\frac{n-3}{2}} \d\tau \d\alpha=\\
         &\omega_{n-1}\sum_{j=0}^{n-2}\binom{n-2}{j}\int_{0}^{\lambda} \int_{-\sqrt{1-\frac{s^2}{t^2}}}^{0} p(t+p\alpha)t^{n-j-2}p^j\left(\alpha + \frac{\mu}{\sqrt{1-\tau^2}}\right)^j(1-\tau^2)^{\frac{n-3}{2}}\d\tau \d\alpha \\
         & =\omega_{n-1}\sum_{j=0}^{n-2}\binom{n-2}{j}p^{j+1}t^{n-j-1}\sum_{k=0}^j\binom{j}{k}\mu^{k}\int_{0}^{\lambda}\alpha^{j-k}\d r \int_{-\sqrt{1-\frac{s^2}{t^2}}}^{0} (1 - \tau^2)^{\frac{n-k-3}{2}} \d\tau +\\
         &\omega_{n-1}\sum_{j=0}^{n-2}\binom{n-2}{j}p^{j+2}t^{n-j-2}\sum_{k=0}^j\binom{j}{k}\mu^{k}\int_{0}^{\lambda}\alpha^{j-k+1}\d r \int_{-\sqrt{1-\frac{s^2}{t^2}}}^{0} (1 - \tau^2)^{\frac{n-k-3}{2}} \d\tau= \\
         &\omega_{n-1}\sum_{j=0}^{n-2}\binom{n-2}{j}p^{j+1}t^{n-j-1}\sum_{k=0}^j\binom{j}{k}\frac{1}{j-k+1}\mu^{k}\lambda^{j-k+1} \int_{-\sqrt{1-\frac{s^2}{t^2}}}^{0} (1 - \tau^2)^{\frac{n-k-3}{2}} \d\tau +\\
         &\omega_{n-1}\sum_{j=0}^{n-2}\binom{n-2}{j}p^{j+2}t^{n-j-2}\sum_{k=0}^j\binom{j}{k}\frac{1}{j-k+2}\mu^{k}\lambda^{j-k+2} \int_{-\sqrt{1-\frac{s^2}{t^2}}}^{0} (1 - \tau^2)^{\frac{n-k-3}{2}} \d\tau.
     \end{align*}
     Thus, for $j=1$,
     \begin{align*}
          \Phi^n_1(\mu  w_s + \lambda v_t; B_6)=\omega_{n-1}\lambda\  t^{n-1}\int_{-\sqrt{1-\frac{s^2}{t^2}}}^{0} (1 - \tau^2)^{\frac{n-3}{2}} \d\tau
     \end{align*}
     and 
     \begin{align*}
         &\Phi^n_j(\mu  w_s + \lambda v_t; B_6)= \frac{\omega_{n-1}}{n-1}t^{n-j}\left[\binom{n-1}{j}\sum_{k=0}^{j-1}\binom{j}{k}\mu^k \lambda^{j-k}\int_{-\sqrt{1-\frac{s^2}{t^2}}}^{0} (1 - \tau^2)^{\frac{n-k-3}{2}} \d\tau \right.\\
         &\quad +\left.\frac{1}{n}\binom{n}{j}\sum_{k=0}^{j-2}\binom{j}{k}(j-k-1)\mu^k \lambda^{j-k}\int_{-\sqrt{1-\frac{s^2}{t^2}}}^{0} (1 - \tau^2)^{\frac{n-k-3}{2}} \d\tau\right]\\
         &=\frac{\omega_{n-1}}{n-1}\frac{1}{n}\binom{n}{j}t^{n-j}\sum_{k=0}^{j-1}\binom{j}{k}(n-k-1)\mu^k \lambda^{j-k}\int_{-\sqrt{1-\frac{s^2}{t^2}}}^{0} (1 - \tau^2)^{\frac{n-k-3}{2}} \d\tau 
     \end{align*}
     for every $2 \leq j \leq n-1$, whereas for $j=n$, we have,
     \begin{align*}
         \Phi^n_n(\mu  w_s + \lambda v_t; B_6)=\frac{\omega_{n-1}}{n-1}\frac{1}{n}\sum_{k=0}^{n-2}\binom{n}{k}(n-k-1)\mu^{k}\lambda^{n-k} \int_{-\sqrt{1-\frac{s^2}{t^2}}}^{0} (1 - \tau^2)^{\frac{n-k-3}{2}} \d\tau
     \end{align*}
     Hence for every $1 \leq j \leq n$, the measure
   
     \[\eta_j =\Phi^n_j(\mu  w_s + \lambda v_t;\cdot)|_{B_6}\]
    is a finite Borel measure. Now, let
     \begin{align*}
         \pi: \R^n &\rightarrow \R\\
         x &\rightarrow \pi(x)= x_n.
     \end{align*}
     Setting $\Sigma{'} = (-\sqrt{t^2-s^2},0)$, then for each $x_n \in \Sigma{'}$, we have
     \[B_6 \cap \pi^{-1}(x_n) = \sqrt{t^2-x_n^2}\s^{n-2}\]
    By Theorem \ref{disintegration_theorem} applied to $\eta_j$ and to the projection $\pi$, there exists a family of measures $(\eta_{j,{x_n}})_{x_n \in \Sigma{'}}$ defined for $(\pi_{\#}\eta_j)$-almost every $x_n$ such that $\eta_{j,x_n}$ is concentrated on $B_6 \cap \pi^{-1}(x_n)$, and
     \begin{align*}
        \int_{B_6}\varphi(x)&\d\Phi^n_j(\mu  w_s + \lambda v_t;x) 
        =
          \int_{\Sigma^{'}}\int_{\sqrt{t^2-x_n^2}\mathbb{S}^{n-2}} \varphi(x^{'},x_n)\d(\eta_{j,x_n})(x{'}) \d (\pi_{\#}\eta_j)(x_n)
     \end{align*}
     for every bounded Borel function $\varphi: \R^n \rightarrow \R$. 
     Since the function $\mu  w_s + \lambda v_t$ is $\SO(n-1)$ invariant on $\sqrt{t^2-x_n^2}\mathbb{S}^{n-2}$ then, for every $x_n \in \Sigma{'}$, the corresponding measure $\eta_{j,{x_n}}$ must be scalar multiple of $\hm^{n-2}$ on this sphere. This implies that there exist a function $L_{n,j}$ depending on $x_n$ such that 
     \begin{align}\label{disintegration_hessian_B_6}
         \int_{B_6} \varphi(x)\d\Phi^n_j(\mu  w_s + \lambda v_t;x)= \int_{\Sigma^{'}}L_{n,j}(x_n) \int_{\sqrt{t^2-x_n^2}\ \mathbb{S}^{n-2}}\varphi(x^{'},x_n) \d\hm ^{n-2}(x^{'})\d \pi_{\#}\eta_j(x_n),
    \end{align}
     and since for every generating Borel set of $\R$, which is given by $(-\infty,c)$ for some $c \in \R$, we have
     \begin{align*}
         (\pi_{\#}\eta_j)(Y) = \Phi^n_j(\mu  w_s + \lambda v_t;B_6 \cap \pi^{-1}(Y)),    
    \end{align*}
     then by computing the parallel set of the function $\mu  w_s + \lambda v_t$ at $B_6 \cap \pi^{-1}(Y)$, it follows that 
      \begin{align*}
         (\pi_{\#}\eta_1)(Y)&=\omega_{n-1}\lambda\  t^{n-1}\int_{-\sqrt{1-\frac{s^2}{t^2}}}^{0} \mathds{1}_Y(t \tau)(1 - \tau^2)^{\frac{n-3}{2}} \d\tau\\
          & = \omega_{n-1}\lambda\  t\int_{-\sqrt{t^2-s^2}}^{0} \mathds{1}_Y(x_n)(t^2 - x_n^2)^{\frac{n-3}{2}} \d x_n
     \end{align*}
     and, for every $2 \leq j \leq n-1$, the measure $\pi_{\#}\eta_j$ at $Y \subset \Sigma{'}$ is given by 
     \begin{align*}
        & \frac{\omega_{n-1}}{n-1}\frac{1}{n}\binom{n}{j}t^{n-j}\sum_{k=0}^{j-1}\binom{j}{k}(n-k-1)\mu^k \lambda^{j-k}\int_{-\sqrt{1-\frac{s^2}{t^2}}}^{0} \mathds{1}_Y(t\tau)(1 - \tau^2)^{\frac{n-k-3}{2}} \d\tau =\\
         &\frac{\omega_{n-1}}{n-1}\frac{1}{n}\binom{n}{j}t^{2-j}\sum_{k=0}^{j-1}\binom{j}{k}(n-k-1)\mu^k \lambda^{j-k}\int_{-\sqrt{t^2-s^2}}^{0}\mathds{1}_Y(x_n) (t^2 - x_n^2)^{\frac{n-k-3}{2}} \d x_n 
     \end{align*}
      whereas for $j=n$, we have that
     \begin{align*}
         (\pi_{\#}\eta_n)&(Y)=\frac{\omega_{n-1}}{n-1}\frac{1}{n}\sum_{k=0}^{n-2}\binom{n}{k}(n-k-1)\mu^{k}\lambda^{n-k} \int_{-\sqrt{1-\frac{s^2}{t^2}}}^{0} \mathds{1}_Y(t\tau) (1 - \tau^2)^{\frac{n-k-3}{2}} \d\tau\\
         &=\frac{\omega_{n-1}}{n-1}\frac{1}{n}\sum_{k=0}^{n-2}\binom{n}{k}(n-k-1)t^{2-n}\mu^{k}\lambda^{n-k}  \int_{-\sqrt{t^2-s^2}}^{0}\mathds{1}_Y(x_n) (t^2 - x_n^2)^{\frac{n-k-3}{2}} \d x_n
     \end{align*}
     This implies that the measure $\pi_{\#}\eta_j$ is absolutely continuous with respect to the Lebesgue measure. After substituting for $\varphi = 1$ in (\ref{disintegration_hessian_B_6}) and comparing with the total measure of $\Phi^n_j(\mu  w_s + \lambda v_t;\cdot)|_{B_6}$, we get that $L_{n,j}(x_n) = \frac{(t^2-x_n^2)^{\frac{2-n}{2}}}{\omega_{n-1}}$ for every $x_n \in \R$. Thus, for $j=1$, we have
    \begin{align*}
         \int_{B_6} \xi(|x|)x_n&\d\Phi^n_1(\mu  w_s + \lambda v_t;x) \\
         &=\omega_{n-1}\lambda t\int_{\Sigma^{'}}(t^2 - x_n^2)^{\frac{n-3}{2}}x_n L_{n,1}(x_n)\int_{\sqrt{t^2-x_n^2} \sn}\xi(|x|) \d\hm ^{n-2}(x^{'})\d x_n\\
         &=\omega_{n-1} \lambda t^{n} \xi(t) \int_{-\sqrt{1-\frac{s^2}{t^2}}}^{0} (1 - u^2)^{\frac{n-2}{2}}u \d u \\
         &= \frac{\omega_{n-1}}{n-1}\lambda\ \xi(t)\left[-t^{n}  +  t s^{n-1} \right]
    \end{align*}
    and for every $2 \leq j \leq n-1$, we have
    \begin{align*}
         \int_{B_6} &\xi(|x|)x_n\d\Phi^n_j(\mu  w_s + \lambda v_t;x) \\
         &=\frac{1}{n}\binom{n}{j}\frac{\omega_{n-1}}{n-1} t^{n-j+1} \xi(t)\sum_{k=0}^{j-1}\binom{j}{k}(n-k-1)\mu^k \lambda^{j-k} \int_{-\sqrt{1-\frac{s^2}{t^2}}}^{0} u(1 - u^2)^{\frac{n-k-3}{2}} \d u\\
         & =\frac{\omega_{n-1}}{n-1} t^{n-j+1} \xi(t)\frac{1}{n}\binom{n}{j}\sum_{k=0}^{j-1}\binom{j}{k}\mu^k \lambda^{j-k} \left( -1 + \frac{s^{n-k-1}}{t^{n-k-1}} \right), 
    \end{align*}
    and for $j =n$, we have
    \begin{align*}
         \int_{B_6} \xi(|x|)x_n&\d\Phi^n_n(\mu  w_s + \lambda v_t;x)\\
         & = \frac{\omega_{n-1}}{n-1}\frac{1}{n}\sum_{k=0}^{n-2}\binom{n}{k}(n-k-1)\mu^{k}\lambda^{n-k}  \xi(t) t \int_{-\sqrt{1-\frac{s^2}{t^2}}}^{0} u(1 - u^2)^{\frac{n-k-3}{2}} \d u \\
         &=- \frac{\omega_{n-1}}{n-1}\frac{1}{n}\sum_{k=0}^{n-2}\binom{n}{k}\mu^{k}\lambda^{n-k}\left[t -  s^{n-k-1}t^{k+2-n}\right]\xi(t).
    \end{align*}
    \item For the last set of singular points $B_7=\{x \in \R^n : |x|_{n-1} = s,|x|=t, x_n < 0\}$, we have by (\ref{singular_point_addition}) that 
    \begin{align*}
        \partial (\mu \ w_s + \lambda \ v_t)(x) =  \Bigl\{  \alpha \frac{x}{|x|}: \alpha \in [0,\lambda]\Bigl\} +  \Bigl\{  (\beta \frac{x_1}{|x|_{n-1}},\dots,\beta \frac{x_{n-1}}{|x|_{n-1}},0) : \beta \in [0,\mu]\Bigl\}
    \end{align*}
    for every $x \in B_7$. This implies that
    \begin{align*}
        &P_p(\mu  w_s + \lambda v_t, B_7)=\\
        &\Bigl\{ x + p\beta \frac{(x - x_n e_n)}{|x|_{n-1}} + p \alpha \frac{x}{|x|} :  |x|= t, |x|_{n-1} = s,  x_n < 0, \alpha \in [0,\lambda], \beta \in [0,\mu] \Bigl\}=\\
         &\Bigl\{ (t + p\alpha)(\tau e_n + \sqrt{1-\tau^2}v) + p\beta v: v \in \sn\cap e_n^{\perp}, \tau=-\sqrt{1-\frac{s^2}{t^2}}, \alpha \in [0,\lambda], \beta \in [0,\mu]\Bigl\}
    \end{align*}
   and thus, the $n$-dimensional Hausdorff measure of the parallel set of $B_7$ is given by
    \begin{align*}
         &\omega_{n-1}\int_{0}^{\lambda}\int_{0}^{\mu}p^{2}\sqrt{1-\frac{s^2}{t^2}}\left(s + p \left(\frac{\alpha s}{t} + \beta\right)\right)^{n-2}\d\beta\d\alpha= \\
         & \omega_{n-1}\sum_{j=0}^{n-2}\binom{n-2}{j}\int_{0}^{\lambda}\int_{0}^{\mu}p^{j+2}\sqrt{1-\frac{s^2}{t^2}}s^{n-j-2}\left(\frac{\alpha s}{t} + \beta\right)^{j}\d\beta\d\alpha= \\
         & \omega_{n-1}\sum_{j=0}^{n-2}p^{j+2}\binom{n-2}{j}\sqrt{t^2-s^2}\sum_{k=0}^j \binom{j}{k}\frac{s^{n-k-2}}{t^{j-k+1}}\int_{0}^{\lambda}\alpha^{j-k}\d\alpha\int_{0}^{\mu}\beta^{k}\d\beta=\\
         &\omega_{n-1}\sum_{j=0}^{n-2}p^{j+2}\binom{n-2}{j}\sqrt{t^2-s^2}\sum_{k=0}^j \binom{j}{k}\frac{s^{n-k-2}}{t^{j-k+1}}\frac{\lambda^{j-k+1}}{j-k+1}\frac{\mu^{k+1}}{k+1}
    \end{align*}
    This implies that 
    \begin{align*}
        \Phi^n_j(\mu  w_s + \lambda v_t; B_7)= \frac{\omega_{n-1}}{n-1}\frac{1}{n}\binom{n}{j} \sqrt{t^2-s^2}\sum_{k=0}^{j-2}\binom{j}{k+1} \frac{s^{n-k-2}}{t^{j-k-1}} \lambda^{j-k-1}\mu^{k+1}
    \end{align*}
    for every $2 \leq j \leq n$.
    Now since 
    \[B_7  =\{x \in \R^n: |x|_{n-1} = s, |x|= t, x_n< 0\} = s \s^{n-2} \times \{-\sqrt{t^2-s^2}\},\]
    the function $x \mapsto \xi(|x|)x_n$ is constant on $B_7$. Consequently,
    \begin{align*}
        \int_{B_7}\xi(|x|)&x_n\d\Phi^n_j(\mu  w_s + \lambda v_t; x) = -\sqrt{t^2-s^2}\xi(t)\Phi^n_j(\mu  w_s + \lambda v_t; B_7).
    \end{align*}
    Substituting the value of $\Phi^n_j(\mu  w_s + \lambda v_t; B_7)$ gives,
    \begin{align*}
        \int_{B_7}\xi(|x|)&x_n\d\Phi^n_j(\mu  w_s + \lambda v_t; x)\\
        &= 
         - \frac{\omega_{n-1}}{n-1}\frac{1}{n}\binom{n}{j} \xi(t)\sum_{k=0}^{j-2}\binom{j}{k+1} \lambda^{j-k-1}\mu^{k+1}  \left[\frac{s^{n-k-2}}{t^{j-k-3}} -\frac{s^{n-k}}{t^{j-k-1}} \right] \\
        &= - \frac{\omega_{n-1}}{n-1}\frac{1}{n}\binom{n}{j} \xi(t)\sum_{k=1}^{j-1}\binom{j}{k} \lambda^{j-k}\mu^{k}  \left[\frac{s^{n-k-1}}{t^{j-k-2}} -\frac{s^{n-k+1}}{t^{j-k}} \right]
    \end{align*}
    for every $2 \leq j \leq n$.
        \end{enumerate}
   Together with the region
   \[Z_{s<t}=\{x_n>0,\ |x|<s\}\cup \{x_n<0,\ |x|_{n-1}<s,\ |x|<t\}
   \]
   on which $\mu w_s + \lambda v_t$ is is identically zero, and with the hyperplane $H=\{x_n=0\}$, the sets
   \[A_1,\ldots,A_5,\quad B_1,\ldots,B_7,\quad Z_{s<t},\quad H
   \]
   form a disjoint partition of $\R^n$. The sets $Z_{s<t}$ and $H$ do not contribute to the integral. Therefore, collecting the preceding contributions by grouping the terms according to their powers in $\lambda$ and $\mu$. The terms containing only  $\lambda^j$ cancel. The only remaining term containing only $\mu^j$ is
   \[\frac{\omega_{n-1}}{n-1}\frac{1}{n}\binom{n}{j}\mu^j s^{n-j+1}\xi(s).\]
    For $1\leq k\leq j-1$, the terms containing $\lambda^{j-k}\mu^k$ combine into
    \[\frac{\omega_{n-1}}{n-1}\frac{1}{n}\binom{j}{k}\lambda^{j-k}\mu^k s^{n-k+1}\left(t^{k-j}\xi(t)-
    (j-k)\int_t^\infty r^{k-j-1}\xi(r)\,\d r\right).
    \]
    By the definition of the inverse $\cR$-transform, the expression between parentheses is
    \[\cR^{k-j}\xi(t). \]
    Thus, for $s < t$, we obtain
    \[
    \begin{aligned}
    \int_{\R^n}\xi(|x|)x_n
    \d\Phi_j^n(\mu w_s+\lambda v_t;x)
    &=\frac{\omega_{n-1}}{n-1}\frac1n\binom{n}{j}
    \left[\mu^j s^{n-j+1}\xi(s)\right.\\
    &\qquad\left. + \sum_{k=1}^{j-1}\binom{j}{k}\lambda^{j-k}\mu^k s^{n-k+1}
    \cR^{k-j}\xi(t)
    \right].
    \end{aligned}
    \]
    \item For $t <s$, we divide the domain $\R^n$ into the open $C^2$-regions $C_1,\dots,C_4$ and the singular regions $D_1,\dots,D_4$. In this case, the spherical level \(|x|=t\) lies before the level \(|x|=s\), and the mixed terms are therefore collected at the radius \(s=\max\{s,t\}\). 
    \begin{enumerate}
       \item The function $\mu  w_s + \lambda v_t$ is of class $C^2$ on the set $C_1= \{x \in \R^n:  t < |x| <s, x_n >0\}$, and its Hessian matrix at a point $x$ in this set has $(n-1)$-eigenvalues equal to $\lambda/|x|$ and the last eigenvalue is equal to zero. Hence,
        $$[\Hess  (\mu  w_s + \lambda v_t)(x)]_j=\begin{cases}
            \binom{n-1}{j}\frac{\lambda^j}{|x|^j}  & \text{if}\ 0 \leq j \leq n-1 \\
            0 & \text{if}\ j=n 
        \end{cases}$$
        and for $1 \leq j \leq n-1$,
        \begin{align*}
                \int_{C_1}\xi(|x|)x_n[\Hess (\mu  w_s + \lambda v_t)(x)]_j\d x &=\lambda^j\binom{n-1}{j}\int_{t}^{s} r^{n-j} \xi(r) \d r \int_{\sn \cap {y_n >0}}y_n \d y\\
                &=\lambda^j\binom{n-1}{j}  \frac{\omega_{n-1}}{n-1} \int_{t}^{s} r^{n-j} \xi(r) \d r.
                \end{align*}    
        \item The function $\mu  w_s + \lambda v_t$ is of class $C^2$ on the set $C_2= \{x \in \R^n :  |x| > s, x_n > 0 \}$, and its Hessian matrix at a point $x$ on $C_2$ has $(n-1)$-eigenvalues equal to $(\mu+\lambda)/|x|$ and the last eigenvalue is equal to zero. Hence,
        $$[\Hess  (\mu  w_s + \lambda v_t)(x)]_j=\begin{cases}
            \binom{n-1}{j}\frac{(\mu+\lambda)^j}{|x|^j}  & \text{if}\ 0 \leq j \leq n-1 \\
            0 & \text{if}\ j=n 
        \end{cases}$$
         and for $1 \leq j\leq n-1$,
            \begin{align*}
                \int_{C_2}\xi(|x|)&x_n[\Hess (\mu  w_s + \lambda v_t)(x)]_j\d x\\
                &=(\mu+\lambda)^j\binom{n-1}{j}\int_{s}^{\infty} r^{n-j} \xi(r) \d r \int_{\sn \cap {y_n > 0}}y_n \d y \\
                &=(\mu+\lambda)^j\binom{n-1}{j} \frac{\omega_{n-1}}{n-1} \int_{s}^{\infty} r^{n-j} \xi(r) \d r. 
            \end{align*}
        \item The function $\mu  w_s + \lambda v_t $ is also of class $C^2$ on $ C_3 = \{x \in \R^n :  |x|_{n-1} < s, |x| > t, x_n < 0\}$, and its Hessian matrix at a point $x \in C_3$ has $(n-1)$-eigenvalues equal to $\lambda/|x|$ and the last eigenvalue is equal to zero. Hence,
        $$[\Hess  (\mu  w_s + \lambda v_t)(x)]_j=\begin{cases}
            \binom{n-1}{j}\frac{\lambda^j}{|x|^j}  & \text{if}\ 0 \leq j \leq n-1 \\
            0 & \text{if}\ j=n 
        \end{cases}$$
         and for $ 1 \leq j\leq n-1$, if we denote by $C_3^1 = \{x \in \R^n: t < |x| < s, x_n <0\}$ and $C_3^2= \{x \in \R^n : |x| > s, |x|_{n-1} < s, x_n < 0\}$, then $C_3^1 \cup C_3^2 = C_3$ and
            \begin{align*}
                \int_{C_3}&\xi(|x|)x_n[\Hess (\mu  w_s + \lambda v_t)(x)]_j\d x \\&= \lambda^j \binom{n-1}{j}\left[ \int_{C_3^1} \xi(|x|)\frac{x_n}{|x|^j}\d x + \int_{C_3^2}\xi(|x|) \frac{x_n}{|x|^j}\d x\right]\\
                &=  \lambda^j \binom{n-1}{j}\left[\int_{t}^{s} r^{n-j}\xi(r)\int_{\sn\cap e_n^{\perp}} \int_{-1}^{0}\tau(1-\tau^2)^{\frac{n-3}{2}} \d\tau\d y\d r\right. \\
                &+ \left. \int_{s}^{+\infty} r^{n-j}\xi(r)\int_{\sn\cap e_n^{\perp}} \int_{-1}^{-\sqrt{1-\frac{s^2}{r^2}}}\tau(1-\tau^2)^{\frac{n-3}{2}} \d\tau\d y\d r \right]\\
                &= -\frac{\omega_{n-1}}{n-1}\lambda^j \binom{n-1}{j}\left[ \int_{t}^{s} r^{n-j}\xi(r)\d r + s^{n-1} \int_{s}^{+\infty} r^{1-j}\xi(r)\d r \right].
            \end{align*}
             \item The function $\mu  w_s + \lambda v_t $ is again of class $C^2$ on $C_4= \{x \in \R^n : |x|_{n-1} > s, x_n <0  \}$, and the eigenvalues of its Hessian matrix at a point $x$ in this set are given by  $$[\Hess (\mu  w_s + \lambda v_t)(x)]_j= \begin{cases}
                 \binom{n-2}{j}\left(\frac{\lambda}{|x|} + \frac{\mu}{|x|_{n-1}}\right)^j \\
                 + & \text{if}\  1 \leq j \leq n-2 \\
                 \binom{n-2}{j-1}\left(\frac{\lambda}{|x|} + \frac{\mu}{|x|_{n-1}}\right)^{j-1}\frac{\lambda}{|x|} \\
                 \\
                \left(\frac{\lambda}{|x|} + \frac{\mu}{|x|_{n-1}}\right)^{n-2}\frac{\lambda}{|x|} & \text{if}\  j=n-1 \\
                \\
                0 & \text{if}\  j=n
             \end{cases}$$
           This implies, for $1 \leq j \leq n-2$, that
            \begin{align*}
                &\int_{C_4}\xi(|x|)x_n[\Hess (\mu  w_s + \lambda v_t)(x)]_j\d x \\
                &=\omega_{n-1}\left[\binom{n-2}{j}\sum_{k=0}^{j}\binom{j}{k}\mu^k \lambda^{j-k}\int_{s}^{+\infty}r^{n-j}\xi(r) \int_{-\sqrt{1-\frac{s^2}{r^2}}}^{0} \tau (1-\tau^2)^{\frac{n-k-3}{2}} \d\tau \d r \right. \\
                &\left. +\binom{n-2}{j-1}\sum_{k=0}^{j-1}\binom{j-1}{k}\mu^k \lambda^{j-k}\int_{s}^{+\infty}r^{n-j}\xi(r) \int_{-\sqrt{1-\frac{s^2}{r^2}}}^{0} \tau (1-\tau^2)^{\frac{n-k-3}{2}} \d\tau \d r   \right] \\
                & =-\omega_{n-1} \times \\
                &\left[\binom{n-2}{j}\sum_{k=0}^{j}\binom{j}{k}\frac{\mu^k \lambda^{j-k}}{n-k-1}\left(\int_{s}^{+\infty}r^{n-j}\xi(r)\d r - s^{n-k-1} \int_{s}^{+\infty}r^{k-j+1}\xi(r)\d r \right) + \right.\\& \left. \binom{n-2}{j-1}\sum_{k=0}^{j-1}\binom{j-1}{k}\frac{\mu^k \lambda^{j-k}}{n-k-1} \left(\int_{s}^{+\infty}r^{n-j}\xi(r)\d r - s^{n-k-1} \int_{s}^{+\infty}r^{k-j+1}\xi(r)\d r \right) \right]  \\
                &= - \frac{\omega_{n-1}}{n-1}\binom{n-1}{j}\times \\&\left[\sum_{k=0}^{j}\binom{j}{k}\mu^k \lambda^{j-k}\left(\int_{s}^{+\infty}r^{n-j}\xi(r)\d r - s^{n-k-1} \int_{s}^{+\infty}r^{k-j+1}\xi(r)\d r \right) \right],
            \end{align*}
        and for $j=n-1$, we have 
              \begin{align*}
                \int_{C_4}&\xi(|x|)x_n[\Hess (\mu  w_s + \lambda v_t)(x)]_j\d x \\
                &= \omega_{n-1}\sum_{k=0}^{n-2}\binom{n-2}{k}\mu^k \lambda^{n-k-1}\int_{s}^{+\infty}r\xi(r) \int_{-\sqrt{1-\frac{s^2}{r^2}}}^{0} \tau (1-\tau^2)^{\frac{n-k-3}{2}} \d\tau \d r \\
                &= -\omega_{n-1}\sum_{k=0}^{n-2}\binom{n-2}{k}\frac{\mu^k \lambda^{n-k-1}}{n-k-1}\left(\int_{s}^{+\infty}r\xi(r)\d r - s^{n-k-1} \int_{s}^{+\infty}r^{k-n+2}\xi(r)\d r \right)\\
                &=-\frac{\omega_{n-1}}{n-1}\sum_{k=0}^{n-2}\binom{n-1}{k}\mu^k \lambda^{n-k-1}\left(\int_{s}^{+\infty}r\xi(r)\d r - s^{n-k-1} \int_{s}^{+\infty}r^{k-n+2}\xi(r)\d r \right).
            \end{align*}
          \item For the first set of singular points $D_1= \{x \in \R^n: |x| = t, x_n > 0\}$, set $V_1 = \{x\in\R^n: |x|<s, x_n>0\}$. Then $V_1$ is open and $D_1 \subset V_1$. On $V_1$ we have,
          \[\mu w_s+\lambda v_t=\lambda v_t.
          \]
          Let $E \subset D_1$ be a Borel set. Since $D_1 \subset V_1$, we have $E \subset V_1$. Hence, by the local determination property of Hessian measures (\ref{locally_determined}) on the open set $V_1$,
          \[\Phi_j^n(\mu w_s+\lambda v_t;E)=\Phi_j^n(\lambda v_t;E).
          \]
          Thus,
          \[\Phi^n_j(\mu  w_s + \lambda v_t;\cdot)|_{D_1} = \Phi^n_j(\lambda v_t;\cdot)|_{D_1}\]
          Since $v_t$ is radially symmetric, we have that $\Phi^n_j( v_t,\cdot)|_{D_1} $ is rotation invariant; this implies that it is proportional to $\hm^{n-1}$ restricted to $D_1$. Hence, there exists a constant $C_{n,j,t}$ such that,
            \begin{align*}
                \int_{D_1} \gamma(x) \d\Phi^n_j(v_t;x)= C_{n,j,t} \int_{D_1}\gamma(x)\d\hm ^{n-1}(x)
             \end{align*}
       for every $\gamma \in C_b(\R^n)$. It follows that,
            \begin{align*}
                 \Phi^n_j(v_t;D_1)&= C_{n,j,t}\int_{D_1}\d\hm ^{n-1}(x) \\
                &=\frac{\omega_{n-1}}{2} C_{n,j,t} t^{n-1}, 
            \end{align*}
        and by a similar proof of Lemma \ref{retrieving_densities_for_volumes}, we have
             \begin{align*}
                  \Phi^n_j(v_t;D_1)= \frac{\kappa_{n}}{2}\binom{n}{j}t^{n-j}
             \end{align*}
         and hence
             \begin{align*}
                  C_{n,j,t}=\frac{1}{n}\binom{n}{j}t^{1-j}.
                 \end{align*}
        This implies that
            \begin{align*}
                \int_{D_1}\xi(|x|)x_n\d\Phi^n_j(\mu  w_s + \lambda v_t;x)&= \int_{D_1}\xi(|x|)x_n\d\Phi^n_j(\lambda v_t;x)\\ 
                &= \lambda^j\frac{1}{n}\binom{n}{j}t^{1-j}\int_{D_1} \xi(|x|)x_n \d\hm ^{n-1}(x)\\
                &=\lambda^j\frac{1}{n}\binom{n}{j}t^{n-j+1}\xi(t)\int_{\sn \cap \{x_n > 0\}}x_n \d\hm ^{n-1}(x)\\
                &=\lambda^j\frac{1}{n}\binom{n}{j}\frac{\omega_{n-1}}{n-1}t^{n-j+1}\xi(t).
            \end{align*}
        \item For the second set of singular points $D_2= \{x \in \R^n: |x|= s, x_n > 0\}$. Since $\mu  w_s + \lambda v_t$ is radial in a neighborhood of $D_2$, by local determination (\ref{locally_determined}), the restriction $\Phi^n_j(\mu  w_s + \lambda v_t;\cdot)|_{D_2} $ is proportional to $\hm^{n-1}|_{D_2}$. This implies that there exists a constant $C_{n,j,s}$ such that,
            \begin{align*}
                \int_{D_2} \gamma(x) \d\Phi^n_j(\mu  w_s + \lambda v_t;x)= C_{n,j,s} \int_{D_2}\gamma(x)\d\hm ^{n-1}(x)
            \end{align*}
        for every $\gamma \in C_b(\R^n)$. While,
            \begin{align*}
                \Phi^n_j(\mu w_s + \lambda v_t;D_2)&= C_{n,j,s}\int_{D_2}\d\hm ^{n-1}(x) \\
                &=\frac{\omega_n}{2}C_{n,j,s} s^{n-1}. 
            \end{align*}
        And on the other hand, since
        \begin{align*}
        P_p(\mu w_s + \lambda v_t, D_2) =  \Bigl\{ (s  + \alpha)x : \alpha \in [\lambda p,(\mu+\lambda)p], x\in \sn, x_n > 0\Bigl\},
        \end{align*}
         we have
        \begin{align*}
        \Phi^n_j(\mu  w_s + \lambda v_t;D_2)&= \binom{n}{j}\frac{\kappa_{n}}{2} s^{n-j} \sum_{k=1}^{j}\binom{j}{k}\lambda^{j-k}\mu^k
        \end{align*}
        which implies that 
        \begin{align*}
        C_{n,j,s}=\frac{1}{n}\binom{n}{j}s^{1-j} \sum_{k=1}^{j}\binom{j}{k}\lambda^{j-k}\mu^k,
        \end{align*}
        and thus
            \begin{align*}
                \int_{D_2}&\xi(|x|)x_n\d\Phi^n_j(\mu  w_s + \lambda v_t;x)\\
                &=\frac{1}{n}\binom{n}{j}s^{1-j} \sum_{k=1}^{j}\binom{j}{k}\lambda^{j-k}\mu^k \int_{D_2} \xi(|x|)x_n \d\hm ^{n-1}(x)\\
                &=\frac{1}{n}\binom{n}{j}s^{n-j+1}\xi(s) \sum_{k=1}^{j}\binom{j}{k}\lambda^{j-k}\mu^k\int_{\sn \cap \{x_n > 0\}}x_n \d\hm ^{n-1}(x)\\
                &=\frac{1}{n} \frac{\omega_{n-1}}{n-1} s^{n-j+1}\xi(s) \sum_{k=1}^{j}\binom{j}{k}\lambda^{j-k}\mu^k.
            \end{align*}
        \item For the third set of singular points $D_3 = \{x \in \R^n : |x|=t, x_n < 0\}$, set $V_3= \{x\in\R^n: |x|_{n-1}<s, x_n<0\}$. Then $V_3$ is open and $D_3 \subset V_3$, because $t < s$. On $V_3$, we have
        \[\mu w_s + \lambda v_t = \lambda v_t.\]
        Therefore, by the local determination property (\ref{locally_determined}) and $j$-homogeneity,
        \[\Phi^n_j(\mu  w_s + \lambda v_t;\cdot)|_{D_3} =\Phi^n_j(\lambda v_t;\cdot)|_{D_3} = \lambda^j\Phi^n_j(v_t;\cdot)|_{D_3}.\]
        Since $v_t$ is radially symmetric, this implies that $\Phi^n_j(v_t;\cdot)$ is rotation invariant on this set. Hence, it is proportional to $\hm^{n-1}|_{D_3}$. This impliex that there exists a constant $K_{n,j,t}$ such that,
            \begin{align*}
                \int_{D_3} \gamma(x) \d\Phi^n_j(v_t;x)= K_{n,j,t} \int_{D_3}\gamma(x)\d\hm ^{n-1}(x)
            \end{align*}
         for every $\gamma \in C_b(\R^n)$. It follows that
            \begin{align*}
                \Phi^n_j(v_t;D_3 )&= K_{n,j,t}\int_{D_3}\d\hm ^{n-1}(x) \\
                &=\frac{\omega_n}{2}K_{n,j,t}t^{n-1} 
            \end{align*}
         and by Lemma \ref{retrieving_densities_for_volumes}, we have
            \begin{align*}
                \Phi^n_j(v_t;D_3)=  \frac{\kappa_{n}}{2}\binom{n}{j}t^{n-j}
            \end{align*}
        and hence 
            \begin{align*}
               K_{n,j,t} = \frac{1}{n}\binom{n}{j}t^{1-j}
            \end{align*}
        which implies that
            \begin{align*}
                \int_{D_3}\xi(|x|)x_n\d\Phi^n_j(\mu  w_s + \lambda v_t;x)
                &= \frac{1}{n}\binom{n}{j} \lambda^j 
                t^{1-j} \int_{D_3} \xi(|x|)x_n \d\hm ^{n-1}(x) \\
                &=  \frac{1}{n}\binom{n}{j}\lambda^j t^{n-j+1}\xi(t)\int_{\sn \cap \{x_n < 0\}} x_n \d\hm ^{n-1}(x)\\
                 & = -\frac{1}{n}\frac{\omega_{n-1}}{n-1} \binom{n}{j} \lambda^j t^{n-j+1} \xi(t).
            \end{align*}
        \item For the fourth set of singular points $D_4= \{x \in\R^n: |x|_{n-1} = s, x_n < 0\}$, by (\ref{singular_point_addition}), we have
        \begin{align*}
            \partial( \mu w_s + \lambda v_t)(x) &= \partial (\mu w_s)  + \partial(\lambda v_t)\\
            &=  \Bigl\{  \left(\alpha \frac{x_1}{|x|_{n-1}},\dots,\alpha \frac{x_{n-1}}{|x|_{n-1}},0\right): \alpha \in [0,\mu]\Bigl\} + \lambda \frac{x}{|x|}
        \end{align*}
           for every $ x \in D_4$ and this implies that
     \begin{align*}
         &P_p(\mu  w_s + \lambda v_t,D_4)=\\& \Bigl\{ x + p \alpha \frac{(x - x_n e_n)}{|x|_{n-1}} + p \lambda \frac{x}{|x|}: |x|_{n-1} = s,   x_n = - \sqrt{|x|^2 - s^2}, \alpha \in [0,\mu] \Bigl\}=\\
         &\Bigl\{ (r + p\lambda)(\tau e_n + \sqrt{1-\tau^2}v) + p \alpha v : v \in \sn\cap e_n^{\perp}, r > s,  \tau = -\frac{1}{r}\sqrt{r^2-s^2}, \alpha \in [0,\mu]\Bigl\}\\
         & =\Bigl\{ -\left(1 + p\frac{\lambda}{r}\right)\sqrt{r^2-s^2} e_n + \left(s + p \left(\frac{\lambda s}{r} + \alpha\right)\right)v : v \in \sn\cap e_n^{\perp},  r > s,  \alpha \in [0,\mu]\Bigl\}
     \end{align*}
     for all $p \geq 0$. Hence, the $n$-dimensional Hausdorff measure of the parallel set $D_4$ is given by 
        \begin{align*}
            &\omega_{n-1}\int_{s}^{+\infty}\int_{0}^{\mu} \left[p\lambda\frac{s^2}{r^2\sqrt{r^2-s^2}} + \frac{r}{\sqrt{r^2-s^2}}\right] p \left(s + p \left(\frac{\lambda s}{r} + \alpha\right)\right)^{n-2} \d\alpha\d r\\
            &=\frac{\omega_{n-1}}{n-1}\int_{s}^{+\infty}\left[p\lambda\frac{s^2}{r^2\sqrt{r^2-s^2}} + \frac{r}{\sqrt{r^2-s^2}}\right]\times \\
            &\left[ \left(s + p \left(\frac{\lambda s}{r} + \mu\right)\right)^{n-1} - \left(s + p\frac{\lambda s}{r} \right)^{n-1} \right]\d r\\
            &=\frac{\omega_{n-1}}{n-1}\sum_{j=1}^{n-1}p^{j} \binom{n-1}{j}\sum_{k=1}^{j}\binom{j}{k}s^{n-k-1}  \int_{s}^{+\infty}\left[p\lambda\frac{s^2}{r^2\sqrt{r^2-s^2}} + \frac{r}{\sqrt{r^2-s^2}}\right]\frac{\lambda^{j-k}}{r^{j-k}}\mu^{k} \d r\\
            &=\frac{\omega_{n-1}}{n-1}\sum_{j=1}^{n-1}p^{j+1}\binom{n-1}{j}\sum_{k=1}^{j}\binom{j}{k}  s^{n-k+1}\lambda^{j-k+1}\mu^k  \int_{s}^{+\infty}\frac{1}{r^{j-k+2}\sqrt{r^2-s^2} }\d r   \\
            & \quad +\frac{\omega_{n-1}}{n-1}\sum_{j=1}^{n-1} p^{j} \binom{n-1}{j} \sum_{k=1}^{j}\binom{j}{k}s^{n-k-1}\lambda^{j-k}\mu^k  \int_{s}^{+\infty}\frac{1}{r^{j-k-1}\sqrt{r^2-s^2} }\d r. 
        \end{align*}
     Thus, for $j =1$ we have
    \begin{align*}
        \Phi^n_1(\mu  w_s + \lambda v_t; D_4)&= \omega_{n-1}s^{n-2}\mu\int_{s}^{+\infty}\frac{r}{\sqrt{r^2-s^2}}\d r\\
        &= \omega_{n-1}s^{n-2}\mu\int_{-\infty}^{0}\d x_n
     \end{align*}
     and 
     \begin{align*}
         &\Phi^n_j(\mu  w_s + \lambda v_t; D_4)\\&= \frac{\omega_{n-1}}{n-1} \left[ \binom{n}{j}\frac{1}{n}\sum_{k=1}^{j-1}\binom{j}{k} (j-k) s^{n-k+1}\lambda^{j-k}\mu^k  \int_{s}^{+\infty}\frac{1}{r^{j-k+1}\sqrt{r^2-s^2} }\d r \right.\\
         &\left. \quad +\binom{n-1}{j} \sum_{k=1}^{j}\binom{j}{k}s^{n-k-1}\lambda^{j-k}\mu^k  \int_{s}^{+\infty}\frac{1}{r^{j-k-1}\sqrt{r^2-s^2} }\d r\right]\\
          &= \frac{\omega_{n-1}}{n-1} \left[ \binom{n}{j}\frac{1}{n}\sum_{k=1}^{j-1}\binom{j}{k} (j-k) s^{n-k+1}\lambda^{j-k}\mu^k  \int_{-\infty}^{0}(x_n^2+s^2)^{\frac{k-j-2}{2}}\d x_n \right.\\
         &\left. \quad+\binom{n-1}{j} \sum_{k=1}^{j}\binom{j}{k}s^{n-k-1}\lambda^{j-k}\mu^k  \int_{-\infty}^{0}(x_n^2+s^2)^{\frac{k-j}{2}}\d x_n\right]
     \end{align*}
     for every $2 \leq j \leq n-1$, and for $j =n$
     \begin{align*}
          \Phi^n_{n}(\mu  w_s + \lambda v_t; D_4)&=\frac{\omega_{n-1}}{n-1}\sum_{k=1}^{n-1}\binom{n-1}{k}  s^{n-k+1}\lambda^{n-k}\mu^k  \int_{s}^{+\infty}\frac{1}{r^{n-k+1}\sqrt{r^2-s^2} }\d r\\
          &=\frac{\omega_{n-1}}{n-1}\sum_{k=1}^{n-1}\binom{n-1}{k}  s^{n-k+1}\lambda^{n-k}\mu^k  \int_{-\infty}^{0}(x_n^2+s^2)^{\frac{k-n-2}{2}}\d x_n
     \end{align*}
     Since $\xi$ has bounded support, there exists $R>s$ such that $\xi(r)=0$ for $r\ge R$. Hence the integral over $D_4$ only depends on the restriction of
     \[\sigma_j =\Phi^n_j(\mu  w_s + \lambda v_t;\cdot)|_{D_4} \]
    to the bounded set $D_4\cap\{|x|\le R\}$, on which this measure is finite. Thus we may apply Theorem \ref{disintegration_theorem} to this finite restriction of $\sigma_j$, and we keep the notation $\sigma_j$ for simplicity. Let,
     \begin{align*}
         \pi: \R^n &\rightarrow \R\\
         x &\rightarrow \pi(x)= x_n.
     \end{align*}
     Setting $\Sigma{''} = (-\infty,0)$, then for each $x_n \in \Sigma{''}$, we have
     \[D_4 \cap \pi^{-1}(x_n) = s\s^{n-2}\]
    By Theorem \ref{disintegration_theorem} applied to $\sigma_j$ and to the projection $\pi$, there exists a family of measures $(\sigma_{j,{x_n}})_{x_n \in \Sigma{''}}$ defined for $(\pi_{\#}\sigma_j)$-almost every $x_n$ such that $\sigma_{j,x_n}$ is concentrated on $D_4 \cap \pi^{-1}(x_n)$. Since the function $\mu  w_s + \lambda v_t$ is $\SO(n-1)$ invariant on $s\mathbb{S}^{n-2}$ then, for every $x_n \in \Sigma{''}$, the corresponding measure $\sigma_{j,{x_n}}$ must be scalar multiple of $\hm^{n-2}$ on this sphere. This implies that there exists a function $N_{n,j}$ depending on $x_n$ such that 
     \begin{align}\label{disintegration_hessian_D}
         &\int_{D_4} \varphi(x)\d\Phi^n_j(\mu  w_s + \lambda v_t;x)= \notag \\
         &\int_{\Sigma^{''}} \d  (\pi_{\#}\sigma_j)(x_n)\int_{s\mathbb{S}^{n-2}}N_{n,j}(x_n)\varphi(x^{'},x_n) \d\hm ^{n-2}(x^{'})
    \end{align}
    for every bounded Borel function $\varphi: \R^n \rightarrow \R$, and since for every generating Borel set of $\R$ which is given by $(-\infty,c)$ for some $c \in \R$, we have
     \begin{align*}
         (\pi_{\#}\sigma_j)(Y) = \Phi^n_j(\mu  w_s + \lambda v_t;D_4 \cap \pi^{-1}(Y)),   
    \end{align*}
     then by computing the parallel set of the function $\mu  w_s + \lambda v_t$ at $D_4 \cap \pi^{-1}(Y)$, we retrieve the following 
     \begin{align*}
         \Phi^n_1(\mu  w_s + \lambda v_t;D_4 \cap \pi^{-1}(Y)) = &\omega_{n-1}s^{n-2}\mu\int_{s}^{+\infty}\mathds{1}_Y\left(-\sqrt{r^2-s^2}\right)\frac{r}{\sqrt{r^2-s^2}}\d r \\
         &=\omega_{n-1}s^{n-2}\mu\int_{-\infty}^{0}\mathds{1}_Y(x_n)\d x_n
     \end{align*}
     for $j=1$, and
     \begin{align*}
          &\Phi^n_j(\mu  w_s + \lambda v_t;D_4 \cap \pi^{-1}(Y))=\frac{\omega_{n-1}}{n-1} \times \\
         & \left[ \binom{n}{j}\frac{1}{n}\sum_{k=1}^{j-1}\binom{j}{k} (j-k) s^{n-k+1}\lambda^{j-k}\mu^k  \int_{-\infty}^{0}\mathds{1}_Y(x_n)(x_n^2+s^2)^{\frac{k-j-2}{2}}\d x_n +\right.\\
         &\left. \binom{n-1}{j} \sum_{k=1}^{j}\binom{j}{k}s^{n-k-1}\lambda^{j-k}\mu^k   \int_{-\infty}^{0}\mathds{1}_Y(x_n)(x_n^2+s^2)^{\frac{k-j}{2}}\d x_n \right]
     \end{align*}
     for every $2 \leq j \leq n-1$ and 
     \begin{align*}
          &\Phi^n_n(\mu  w_s + \lambda v_t;D_4 \cap \pi^{-1}(Y))=\\
          &\frac{\omega_{n-1}}{n-1}\sum_{k=1}^{n-1}\binom{n-1}{k}  s^{n-k+1}\lambda^{n-k}\mu^k  \int_{-\infty}^{0}\mathds{1}_Y(x_n)(x_n^2+s^2)^{\frac{k-n-2}{2}}\d x_n 
     \end{align*}
     when $j =n$. This implies that the measure $\pi_{\#}\sigma_j$ is absolutely continuous with respect to the Lebesgue measure. After substituting for $\varphi = 1$ in (\ref{disintegration_hessian_D}) and comparing with the total measure of $\Phi^n_j(\mu  w_s + \lambda v_t;\cdot) \llcorner D_4$, we get that $N_{n,j}(x_n) = \frac{s^{2-n}}{\omega_{n-1}}$ for every $x_n \in \R$. Thus, for $j = 1$ 
    \begin{align*}
         \int_{D_4} \xi(|x|)&x_n\d\Phi^n_1(\mu  w_s + \lambda v_t;x)
         \\&= \omega_{n-1}s^{n-2}\mu \int_{-\infty}^{0}\int_{s\s^{n-2}}\xi(|x|) N_{n,1}(x_n) x_n\d\hm ^{n-2}(x^{'})\d\hm ^1(x_n)\\
         & = \omega_{n-1}s^{n-2}  \mu  \int_{-\infty}^{0} \xi\left(\sqrt{s^2+x_n^2}\right)x_n \d\hm ^1(x_n) \\
         &= -\omega_{n-1} s^{n-2} \mu \int_{s}^{+\infty} r\xi(r) \d r,
    \end{align*}
    and for every $1 \leq j \leq n-1$
    \begin{align*}
         \int_{D_4} \xi(|x|)&x_n\d\Phi^n_j(\mu  w_s + \lambda v_t;x)
         \\
         =&-\frac{\omega_{n-1}}{n-1} \left[\binom{n-1}{j} \sum_{k=1}^{j}\binom{j}{k}\lambda^{j-k}\mu^k s^{n-k-1} \int_{s}^{+\infty}r^{-j+k+1}\xi(r)\d r  \right.\\
         &\left. + \binom{n}{j}\frac{1}{n}\sum_{k=1}^{j-1}\binom{j}{k} (j-k) \lambda^{j-k}\mu^k s^{n-k+1} \int_{s}^{+\infty}r^{-j+k-1}\xi(r)\d r\right]
    \end{align*}
    and for $j =n$, we have 
    \begin{align*}
        \int_{D_4} \xi(|x|)&x_n\d\Phi^n_n(\mu  w_s + \lambda v_t;x) \\&=-\frac{\omega_{n-1}}{n-1} \sum_{k=1}^{n-1}\binom{n-1}{k}\lambda^{n-k}\mu^k  s^{n-k+1} \int_{s}^{+\infty} r^{k-n}\xi(r) \d r\\
         &=-\frac{\omega_{n-1}}{n(n-1)} \sum_{k=1}^{n-1}\binom{n}{k}(n-k)\lambda^{n-k}\mu^k  s^{n-k+1} \int_{s}^{+\infty} r^{k-n}\xi(r) \d r.
    \end{align*}
    \end{enumerate}
    Together with the region
    \[Z_{t<s}=\{x_n>0,\ |x|<t\}\cup\{x_n<0,\ |x|<t\},
    \]
   on which $\mu w_s+\lambda v_t$ is identically zero, and with the hyperplane
    \[H=\{x_n=0\},\]
    the sets
    \[C_1,\ldots,C_4,\quad D_1,\ldots,D_4,\quad Z_{t<s},\quad H\]
   form a disjoint partition of $\R^n$. The sets $Z_{t<s}$ and $H$ do not contribute to the integral. We collect the preceding contributions by grouping the terms according to their powers in $\lambda$ and $\mu$. The terms containing only $\lambda^j$ cancel. The only remaining term containing only $\mu^j$ is
    \[\mu^j s^{n-j+1}\xi(s).\]
    For \(1\leq k\leq j-1\), the terms containing $\lambda^{j-k}\mu^k$ combine into
    \[\frac{\omega_{n-1}}{n-1}\frac1n\binom{n}{j}\binom{j}{k}\lambda^{j-k}\mu^k s^{n-k+1}\left(s^{k-j}\xi(s)-(j-k)\int_s^\infty r^{k-j-1}\xi(r)\,\d r\right).
    \]
    By the definition of the inverse $\cR$-transform, the expression between parentheses is
    \[\mathcal R^{k-j}\xi(s).\]
    Thus, for $t<s$, we obtain
    \[\begin{aligned}
    \int_{\R^n}\xi(|x|)x_n\d\Phi_j^n(\mu w_s+\lambda v_t;x)&=\frac{\omega_{n-1}}{n-1}\frac1n\binom{n}{j}\left[\mu^js^{n-j+1}\xi(s)\right.\\&\qquad\left.+\sum_{k=1}^{j-1}\binom{j}{k}\lambda^{j-k}\mu^ks^{n-k+1}\cR^{k-j}\xi(s)
    \right].
    \end{aligned}\]
  \end{enumerate}
  The two cases can be written in the same form. Indeed, if $s\leq t$, then $\cR^{k-j}\xi(t)=\cR^{k-j}\xi(\max\{s,t\})$, while if $t<s$, then $\cR^{k-j}\xi(s)=\cR^{k-j}\xi(\max\{s,t\})$. Hence, for every $s,t > 0$,
    \[\begin{aligned}
    z_{j,\xi}(\mu w_s,\lambda v_t)&=\frac{\omega_{n-1}}{n-1}\frac1n\binom{n}{j}\left[\mu^j s^{n-j+1}\xi(s)\right.\\&\qquad\left.+
    \sum_{k=1}^{j-1}\binom{j}{k}\lambda^{j-k}\mu^ks^{n-k+1}\cR^{k-j}\xi(\max\{s,t\})\right]e_n .
    \end{aligned}\]
  \end{proof}	

\end{document}